\documentclass[reqno]{amsart}

\usepackage{amsmath, amssymb, amsfonts, amsbsy, amscd, latexsym, amsthm,enumitem,color}
\usepackage[mathscr]{euscript}
\date{}
\def\begen{\begin{enumerate}[align=left,leftmargin=*,label=(*)]}
\def\begena{\begin{enumerate}[align=left,leftmargin=*,label=(\alph*)]}
 \def\enden{\end{enumerate}}
 \def\begit{\begin{itemize}[align=left,leftmargin=*]}
  \def\endit{\end{itemize}}
\makeatletter
\newcommand{\leqnomode}{\tagsleft@true}
\newcommand{\reqnomode}{\tagsleft@false}
\makeatother
                         
         \def\nrs{;\\ \noalign{\medskip}}
\def\nr{\nn\\}

\def\ser#1#2{{#1}_{#2}(\vec{x})}

\def\im#1{{\rm Im}(#1)}

\def\C{\mathbb{C}}
\def\Z{\mathbb{Z}}
\def\N{\mathbb{Z}_{\ge0}}

\def\ds{\displaystyle}
\def\nn{\nonumber}
\def\ds{\displaystyle}

\def\gg#1{\Gamma(\displaystyle{#1})}
\def\nn{\nonumber}

\def\Z{{\mathbb Z}}
\def\C{{\mathbb C}}

\def\vx{\vec{x}}
\def\vw{\vec{w}}

\newcommand{\JJ}{J}
\newcommand{\LL}{L}
\newcommand{\MM}{M}

\newtheorem{Theorem}{Theorem}[section]
\newtheorem{Proposition}[Theorem]{Proposition}

\newtheorem{Lemma}[Theorem]{Lemma}

\theoremstyle{definition}
\newtheorem{Definition}[Theorem]{Definition}

\theoremstyle{definition}
\newtheorem{Correspondence}[Theorem]{Correspondence}

\theoremstyle{remark}
\newtheorem{Remark}[Theorem]{Remark}

\numberwithin{equation}{section}

\title{Coxeter group actions and limits of hypergeometric series}

\author{R.M. Green \and Ilia D. Mishev \and Eric Stade}

\keywords{Hypergeometric series, Coxeter groups}\subjclass[2010]{33C20, 20F55}

\begin{document}

\begin{abstract}

In this paper, we use combinatorial group theory and a limiting process to connect various types of  hypergeometric series, and of relations among such series.

We begin with a set $S$ of 56 distinct translates of a certain function $M$, which takes the form of a Barnes integral, and is expressible as a sum of two very-well-poised $_9F_8$ hypergeometric series of unit argument. We consider a known, transitive action of the Coxeter group $W(E_7)$ on this set. We show that, by removing from $W(E_7)$ a particular generator, we obtain a subgroup that is isomorphic to $W(D_6)$, and that acts intransitively on $S$, partitioning it into three orbits, of sizes 32, 12, and 12 respectively.

Taking certain limits of the $M$ functions in the first orbit yields a set of 32 $J$ functions, each of which is a sum of two Saalsch\"utzian $_4F_3$ hypergeometric series of unit argument.  The original action of $W(D_6)$ on the $M$ functions in this orbit is then seen to correspond to a known action of this group on this set of $J$ functions. 

In a similar way, the image of each of the size-12 orbits, under a similar limiting process, is a set of 12 $L$ functions that have been investigated in earlier works.  In fact, these two image sets are the same. 

The limiting process is seen to preserve distance, except on pairs consisting of one $M$ function from each size-12 orbit.

Finally, each known three-term relation among the $J$ and $L$ functions is seen to be obtainable as a limit of a known three-term relation among the $M$ functions. 

\end{abstract}

\maketitle

\section{Introduction}

The  hypergeometric series
\eqnarray \label{fdef} &&F(a,b;c;z) =1+\frac{a\,b}{1! c}\,z\nn\\&&+\frac{a(a+1)b(b+1)}{ 2!  c(c+1)}\,z^2
+\frac{a(a+1)(a+2)b(b+1)(b+2)}{3! c(c+1)(c+2)}\,z^3+\cdots\endeqnarray was introduced and studied by Gauss \cite{Ga} in 1821. Gauss derived many properties of this series,  and demonstrated its relationship to a wide array of elementary and special functions.  The ``Gauss function,'' as \eqref{fdef} is now known,  quickly became ubiquitous in mathematics and the physical sciences.

 The theory of {\it generalized hypergeometric series} -- series like (\ref{fdef}), but with arbitrary numbers of  numerator  and denominator  parameters -- began to take form in the latter part of the 1800's.  (The series (\ref{fdef}) has two numerator parameters, $a$ and $b$, and one denominator parameter, $c$.)  From this period through the early part of the 1900's,   properties of -- and, especially, relations among -- these generalized series were studied extensively.  (See \cite{Ba}, \cite{Bar1}, \cite{Bar2}, \cite{T},  \cite{Wh1}, \cite{Wh2}, and \cite{Wh3}, to name just a few.)

More recently, generalized hypergeometric series -- particularly those with {\it unit argument}, meaning $z=1$ -- have figured prominently in various other contexts.  For example, these series are relevant  to atomic and molecular physics: in the calculation of multiloop Feynman integrals (see \cite[Chapters 8, 9, and 11]{Groz}); as $3$-$j$ and $6$-$j$ coefficients in angular momentum theory (cf. \cite[Sections 2.7 and 2.9]{Dra}); and so on. These series have also proved quite fundamental to the theory of automorphic forms.  Among other things, they arise frequently in explicit calculations of archimedean zeta integrals for automorphic $L$ functions.  See, for example, the pioneering work of Bump \cite{Bu}, and such subsequent works as \cite{St1}, \cite{St2}, \cite{St3}, \cite{St4}, and \cite{ST}.

Generalized hypergeometric series have become sufficiently universal that it is now typical to drop the adjective ``generalized,'' and we will follow this convention from now on.

Recently, combinatorial group theory has been applied quite successfully to the study of {relations} -- especially two-term and three-term relations --  among hypergeometric series of unit argument (and related entities, such as basic hypergeometric series).  See, for example, \cite{BLS}, \cite{BRS}, \cite{GMS}, \cite{LV1}, \cite{LV2}, \cite{M1}, and \cite{R}.
A number of these relations had been examined much earlier, by Thomae \cite{T}, Whipple \cite{Wh1}, \cite{Wh2}, \cite{Wh3},  and others.   But the group-theoretic framework for them was new.  This framework further served to elucidate a variety of new two-term and three-term relations, and  to provide a cohesive framework binding these new relations to the more classical ones.

It is such a framework that we wish to expand upon in the present work.  Here, we will use Coxeter group theory, together with a certain limiting process, to connect hypergeometric series of different orders. (The order of a hypergeometric series is the pair $(p,q)$ indicating the number of numerator and denominator parameters, respectively, of the series.)

Specifically:  we begin with a certain Barnes integral $M(\vec{w})$, expressible as a linear combination of two very-well-poised ${}_9F_8$ hypergeometric series of unit argument.  (By ``${}_9F_8$ hypergeometric series,'' we mean one with nine numerator parameters and eight denominator parameters. The meaning of ``very-well-poised,'' and of various other  technical terms introduced in this section, will be explained in the next section.) Here, $\vw$ is a vector belonging to a certain affine hyperplane $W$ in $\C^8$
(see (\ref{Wdef}) for the definition of $W$). 
We recall, as is described in \cite{LV2} and  \cite{Roy}, that two-term and three-term relations involving the function $M$ can be described in terms of the transitive action of a certain Coxeter group $H$, isomorphic to the Weyl group $W(E_7)$, on the parameters of $M$.  In particular,

\begit

\item (two-term relations) $M(\tau \vw)=M(\vw)$ for all $\tau$ in a certain subgroup $G$ of $H$, isomorphic to the Weyl group $W(E_6)$;
\item (three-term relations) If $\sigma$, $\tau$, and $\mu$ are in distinct right cosets of $G$ in $H$, then there is a linear relation among $M(\sigma \vw)$, $M(\tau\vw)$, and $M(\mu\vw)$.  (The coefficients in this relation are rational combinations of gamma functions, whose arguments depend on $\vw$.)

\endit

We show that a certain subgroup $Q$ of $H$, isomorphic to the Weyl group $W(D_6)$ (and obtained through ``removal'' of one of the generators of $H$), acts {\it intransitively} on the set $\{M(\tau \vw)\colon \tau\in H\}$, and in fact partitions this set into three orbits.  The action of $Q$ on each of these orbits is then seen to mirror known actions of $W(D_6)$ on certain combinations, denoted $J$ and $L$, of Saalsch\"utzian ${}_4F_3$ hypergeometric series, of unit argument.  Moreover, these $J$ and $L$ functions may be obtained from the $M$ function by letting certain parameters of the latter become infinite. This limiting process further transforms $M$-function relations into relations among $J$ functions, relations among $L$ functions, or ``mixed'' relations entailing both $J$ and $L$ functions,  depending on the orbits in which the original $M$ functions sit.  Finally, the limiting process preserves certain prescribed metrics on the relevant spaces of functions, except in certain cases where two $M$ functions  of   different ``colors'' are both mapped to $L$ functions.

The structure of our paper is as follows.  In Section 2, we recall basic definitions and formulas for  hypergeometric series and Barnes integrals in general, and for the particular functions $M$, $J$, and $L$ referenced above.   In Section 3 we assemble some results, previously developed in \cite{Roy}, concerning the action of the Coxeter group $H\cong W(E_7)$ on sets of $M$ functions.  In Section 4 we recall similar results, previously obtained in \cite{FGS}, \cite{GMS}, and \cite{M1}, for the functions $\JJ(\vec{x})$ and 
$\LL(\vec{x})$.  In Section 5, we develop a correspondence between the $M$ functions  of Section 3 and the $J/L$ functions  of Section 4, by invoking Coxeter group theory.  In Section 6, we show that the correspondence of Section 5 preserves certain metrics, except in some cases where two $M$ functions both map to $L$ functions.  In Section 7, we show that the correspondence of Section 5 can also be obtained by taking certain limits of the $M$ functions. In Section 8, we show that our correspondence not only maps $M$ functions to $J/L$ functions, but also takes known three-term relations among the former to analogous, known relations among the latter. Finally, in our Appendix, we provide an explicit characterization of our correspondence.\label{sec1}

\section{Basic definitions and notation}
\label{sec2}

If $f$ is a function of a single complex variable, we will use the following 
shorthand notations:
\begin{equation}
\label{shorthand1}
f[x_1,x_2,\ldots,x_n]
=\prod_{i=1}^{n}f(x_i),
\end{equation}
\begin{equation}
\label{shorthand2}
f\bigg[
{x_1,x_2,\ldots,x_n
\atop
y_1,y_2,\ldots,y_m}
\bigg]
=\frac{f[x_1,x_2,\ldots,x_n]}
{f[y_1,y_2,\ldots,y_m]}.
\end{equation}We also introduce the following notation for finite sequences of complex numbers:
\begin{equation}x\pm(x_1,x_2,\ldots,x_n)=x\pm x_1, \,x\pm x_2,\,\ldots ,\,x\pm x_n,\label{shorthand3}\end{equation}and 
\begin{align}x\pm((x_1,x_2,\ldots,x_n))=&x\pm x_1\pm x_2,\, x\pm x_1\pm x_3,\,\ldots ,\,x\pm x_1\pm x_n,\nr&x\pm x_2\pm x_3,x\pm x_2\pm x_4,\,\ldots 
,\,x\pm x_{n-1}\pm x_n,\label{shorthand4}\end{align}where the $\pm$ will either always denote $+$ or
always denote $-$ in each of its appearances in either
(\ref{shorthand3}) or (\ref{shorthand4}).  So, in particular, we have
\begin{equation}
\label{shorthand5}
f[x\pm(x_1,x_2,\ldots,x_n)]
=\prod_{i=1}^{n}f(x\pm x_i),
\end{equation}
and
\begin{equation}
\label{shorthand6}
f[x\pm((x_1,x_2,\ldots,x_n))]
=\prod_{1\leq i < j \leq n}f(x\pm x_i \pm x_j).
\end{equation}

Next, let $a \in \mathbb{C}$ and $n\in\N$. The Pochhammer symbol
$(a)_n$ is defined by
\begin{equation*}
(a)_n=\left\{ \begin{array}{ll}
a(a+1)\cdots(a+n-1)  & \text{if }n>0,\\
1, &\text{if } n=0.
\end{array} \right.
\end{equation*}
From the functional equation
$s \Gamma(s)=\Gamma(s+1)$ for the gamma function,
we have 
$$(a)_n=\frac{\Gamma(a+n)}{\Gamma(a)},$$
provided
$a \notin \{0, -1, -2,\ldots\}$. 

Let $p$ and $q$ be nonnegative integers and let
$a_1,a_2,\ldots,a_p,b_1,b_2,\ldots,b_q, z
\in \mathbb{C}$.
The hypergeometric series of type ${}_pF_q$ is defined by
\begin{equation*}
\label{210}
{}_pF_q \left[ \genfrac{} {}{0pt}{}{a_1,a_2,\ldots,a_p;}{
 b_1,b_2,\ldots,b_q;} z\right] =
\sum_{n=0}^{\infty} \frac{(a_1)_n(a_2)_n \cdots
(a_p)_n}{n!(b_1)_n(b_2)_n \cdots (b_q)_n}z^n.
\end{equation*}
 To avoid poles, we assume that no denominator parameter
$b_1,b_2,\ldots,b_q$ is a negative integer or zero. If a numerator
parameter $a_1,a_2,\ldots,a_p$ is a negative integer or zero, the
series 
{\it terminates}, as it has only finitely many nonzero terms.
When $z=1$, we say that the series is of {\it unit argument}, or {\it of type}
${}_pF_q(1)$.

In this paper, we will be interested in the case 
$p=q+1$.
The series of type
${}_{q+1}F_q$ converges absolutely if $|z|<1$, or if
$|z|=1$ and $\textrm{Re}(\sum_{i=1}^qb_i-\sum_{i=1}^{q+1}a_i)>0$ (see
\cite[p.\ 8]{Ba}).
If $\sum_{i=1}^qb_i-\sum_{i=1}^{q+1}a_i=1$, the
series is called {\it Saalsch\"utzian}, or {\it balanced}. If
$1+a_1=b_1+a_2=\cdots=b_q+a_{q+1}$, the series is called 
{\it well-poised}.
A well-poised series such that $a_2=1+\frac{1}{2}a_1$ is called
{\it very-well-poised}.

Let $\vec{x}=(A,B,C,D,E,F,G)^T$ belong to the affine hyperplane
\begin{equation}V=\{(A,B,C,D,E,F,G)^T \in \mathbb{C}^7:
E+F+G-A-B-C-D=1\}.\label{vdef}\end{equation}
The function $\JJ(\vec{x})=\JJ(A;B,C,D;E,F,G)$ 
studied in \cite{FGS,GMS}
is defined by
\begin{align}\label{Jdef}
 \JJ(\vec{x}) =&\frac{{}_4F_3\left[ \genfrac{} {}{0pt}{} {\displaystyle A,B,C,D;}{\ds
E,F,G;}1\right]}{\sin\pi A\,\Gamma[ 1+A-(1,E,F,G),E,F,G]} \nr+\,&\frac{{}_4F_3\left[  \genfrac{} {}{0pt}{}{\displaystyle A,1+A-E,1+A-F,1+A-G;}{
\displaystyle 1+A-B,1+A-C,1+A-D;}1\right]} {\sin\pi A\,\Gamma[A,B,C,D,1+A-(B,C,D)]}
\nr=\,&\frac{1}{\sin\pi A\,\Gamma[A,B,C,D,1+A-(1,E,F,G)]}\nr
\times\,&\biggl({_4}F_3^* \biggl[ \genfrac{} {}{0pt}{}{A,B,C,D;}{ E,F,G;}1\biggr] +{_4}F_3^* \biggl[ \genfrac{} {}{0pt}{}{A,1+A-E,1+A-F,1+A-G;}{ 1+A-B,1+A-C,1+A-D;}1\biggr]\biggr),
\end{align}
and the function $\LL(\vec{x})=\LL(A,B,C,D;E;F,G)$ studied in
\cite{GMS,M1} is defined by
\begin{align}\label{Ldef}
 \LL(\vec{x})=&\frac{{}_4F_3\left[  \genfrac{} {}{0pt}{}{\displaystyle A,B,C,D;}{ \displaystyle
E,F,G;}1\right]}{\sin \pi E\,\Gamma[1-E+(A,B,C,D),E,F,G]} \nr-\,&\frac{{}_4F_3\left[  \genfrac{} {}{0pt}{}{\displaystyle 1+A-E,1+B-E,1+C-E,1+D-E;}{
\displaystyle 2-E,1+F-E,1+G-E;}1\right]} {\sin \pi E\,
\Gamma[A,B,C,D,1-E+(1,F,G)]}
\nr=\,&\frac{1}{ { \sin\pi E\,\Gamma[A,B,C,D,1-E+(A,B,C,D)]} }\biggl({_4}F_3^* \biggl[ \genfrac{} {}{0pt}{}{A,B,C,D;}{ E,F,G;}1\biggr]\nr& -{_4}F_3^* \biggl[ \genfrac{} {}{0pt}{}{\displaystyle 1+A-E,1+B-E,1+C-E,1+D-E; }{
\displaystyle 2-E,1+F-E,1+G-E;}1\biggr]\biggr),
\end{align}
where by definition \begin{align*}
{_4}F_3^* \biggl[ \genfrac{} {}{0pt}{}{A,B,C,D;}{ E,F,G;}1\biggr]=\,& \Gamma\biggl[{A,B,C,D\atop
 E,F,G}\biggr]
{{_4}F_3 \biggl[ \genfrac{} {}{0pt}{}{A,B,C,D; }{ E,F,G;}1\biggr]}\nr=\,&\sum_{k=0}^\infty \frac{1}{k!}\Gamma
\biggl[{k+A,k+B,k+C,k+D\atop 
 k+E,k+F,k+G}\biggr].  \end{align*} 
\begin{Remark}
\label{R210}  
We have
\begin{equation}
\label{JandK}
\JJ(\vec{x})=
\frac{1}{\sin \pi A \, \Gamma(A)}
K(\vec{x}),
\end{equation}
where $K(\vec{x})$ is the function considered in \cite{FGS}.  This renormalization allows
for greater simplicity and uniformity of the three-term relations studied in \cite[Sections 5 and 6]{GMS}.  
We also note that, because all invariances, or two-term relations, for $K$ described in \cite{FGS} preserve the first coordinate  of $(A,B,C,D,E,F,G)^T$, such invariances continue to hold for the function $\JJ$.
\end{Remark}

Alternative notation used for 
$\JJ(A;B,C,D;E,F,G)$ and $\LL(A,B,C,D;E;F,G)$ will be
$\JJ\biggl[{\displaystyle{A;B,C,D;}\atop \displaystyle{E,F,G}}\biggr]$
and
$\LL\biggl[{\displaystyle{A,B,C,D;}\atop \displaystyle{E;F,G}}\biggr]$,
respectively.

The two Saalsch\"utzian ${}_4F_3(1)$ series in 
the definition of the $\JJ$ function are called 
{\it complementary} with respect 
to the parameter $a$, and the two 
Saalsch\"utzian ${}_4F_3(1)$ series in 
the definition of the $\LL$ function are called 
{\it supplementary} with respect 
to the parameter $e$. 

We note that, by \cite[Eq.\ $(7.5.3)$]{Ba}, the $\LL$ function can be
expressed as a very-well-poised ${}_7F_6(1)$ series:
\begin{align}
\label{eq240}
&\LL(A,B,C,D;E;F,G) \nr
\quad&=\frac{\Gamma(1+a)}{ 
\pi\, \Gamma[1+a-(b,c,d,e,f),2+2a-b-c-d-e-f]} \nr
 \times\,& {}_7F_6\left[ \genfrac{} {}{0pt}{}{ a,1+a/2,b,c,d,e,f;
}{ 
a/2,1+a-(b,c,d,e,f);}1\right],
\end{align}
where
$$\displaylines{
a=D+G-E, \quad b= G-A, \quad c=G-B , \cr
d= G-C, \quad e=D , \quad f=1+D-E,}$$
and we require that 
$\textrm{Re}(2+2a-b-c-d-e-f)=\textrm{Re}(F-D)>0$.
The above representation of the $\LL$ function can also
be written as
\begin{equation*}
\LL(A,B,C,D;E;F,G)=\frac{\psi[a;b,c,d,e,f]}{\pi}, 
\end{equation*}
where $\psi$ is the function defined in \cite[Eqs.\ (2.1) and (2.11)]{Wh3}.
Thus, the results concerning the
$\LL$ function can also be interpreted in terms of the very-well-poised
${}_7F_6(1)$ series given in (\ref{eq240}). However, 
we primarily view the $\LL$ function as a linear combination of
two Saalsch\"utzian ${}_4F_3(1)$ series, because this representation
connects the $\LL$ function more directly  to the $\JJ$ function, which is also defined
as such a linear combination.
Further, the above very-well-poised ${}_7F_6(1)$ series converges only for appropriate values of the parameters,
while the Saalsch\"utzian condition $E+F+G-A-B-C-D=1$ guarantees the
convergence of both ${}_4F_3(1)$ series in the definition of the
$\LL$ (and the $\JJ$) function.

Next, let 
 $\vec{w}=(a,b,c,d,e,f,g,h)^T$ belong to the  affine hyperplane
\begin{align}&W=\{(a,b,c,d,e,f,g,h)^T \in \mathbb{C}^8:
2+3a=b+c+d+e+f+g+h\}.\label{Wdef}\end{align}The function $\MM(\vec{w})=\MM(a;b;c,d,e,f,g,h)$ 
studied in \cite{Roy}
is defined by
\begin{align}\label{Mdef}
 &\MM(\vec{w})\nr&=\frac
 {  V(a;b,c,d,e,f,g,h)  -V(2b-a;b,b-a+(c,d,e,f,g,h)
 )} 
 {\sin\pi(b-a)\,
 \Gamma[b,c,d,e,f,g,h,b-a+(c,d,e,f,g,h)]},
\end{align}
where  \begin{align*}
  & V(a;b,c,d,e,f,g,h) 
 = \frac{\pi}{2}
  \Gamma\biggl[{ 1+a,b,c,d,e,f,g,h\atop 1+a-(b,c,d,e,f,g,h)}\biggr]\nr
 &\times   
 {}_9F_8\left[  \genfrac{} {}{0pt}{}
 {\displaystyle a,1+a/2,b,c,\ldots,h;}
 { \displaystyle
 a/2,1+a-b,1+a-c,\ldots,1+a-h;}1\right].
 \end{align*} 

\begin{Remark}
\label{R220}  
The function $\MM$ is called $J$ in \cite{Roy}.
We use the notation $\MM$ here to distinguish from
the function $\JJ(\vec{x})$ defined in
(\ref{Jdef}).
\end{Remark}

Another notation used for 
$\MM(a;b;c,d,e,f,g,h)$ will be
$\MM\biggl[{\displaystyle{a;b;c,d,}\atop \displaystyle{e,f,g,h}}\biggr]$.

The two very-well-poised ${}_9F_8(1)$ series in 
the definition of the $\MM$ function are called 
{\it complementary} with respect 
to the parameter $b$.

In our limit derivations, 
we will make use of Barnes integrals. A Barnes
integral is a contour integral of the form
\begin{equation}
\label{BarInt} \int_t \prod_{i=1}^n\Gamma^{\epsilon_i}(a_i+t)
\prod_{j=1}^m\Gamma^{\epsilon_j}(b_j-t) \, dt,
\end{equation}
where $n,m \in \mathbb{N}; \epsilon_i,\epsilon_j=\pm 1;$ and
$a_i,b_j,t \in \mathbb{C}$. The path of integration is the imaginary
axis, indented, if necessary, to assure that all poles of
$\prod_{i=1}^n\Gamma^{\epsilon_i}(a_i+t)$ are to the left of the
contour and all poles of $\prod_{j=1}^m\Gamma^{\epsilon_j}(b_j-t)$
are to the right of the contour. This path of integration always
exists, provided that, for $1 \leq i \leq n$ and $1 \leq j \leq m$,
we have $a_i+b_j \notin \{0,-1,-2,\ldots\}$ whenever
$\epsilon_i=\epsilon_j=1$.  (An integral like \eqref{BarInt} is also sometimes called a {\it Meijer's $G$-function}.)

From now on, when we write an integral of the form
(\ref{BarInt}), we will always mean a Barnes integral
with a path of integration as just described.

The functions $\JJ(\vec{x}), \LL(\vec{x})$, and
$\MM(\vec{w})$ can be written in terms of Barnes integrals.
It is shown in \cite{FGS}, \cite{M1}, and
\cite{Roy} that
\begin{align}
\label{JBarInt}
&\JJ(A;B,C,D;E,F,G)\nr
&=\frac{1}
{\sin\pi A \, \Gamma[A,1+A-(E,F),B,C,E-(B,C),F-(B,C),G-D]}\nr
&\times \frac{1}{2\pi i}
\int_t 
\Gamma\biggl[{t+G-D,t+1+A -(E,F),
-t+D-G+(B,C),-t\atop -t+D,t+1+A-D}\biggr] \, dt,
\end{align}
\begin{align}
\label{LBarInt}
&\LL(A,B,C,D;E;F,G)=\frac{1}
{\pi \, \Gamma[A,B,C,D,1-E+(A,B,C,D)]}\nr
&\times \frac{1}{2\pi i}
\int_t 
\Gamma\biggl[{ t+(A,B,C,D),
-t+1-E,-t \atop t+(F,G)}\biggr] \, dt,
\end{align}
and
\begin{align}\label{MBarInt}
&\MM(a; b; c, d, e, f, g, h) =  \frac{1} {\Gamma[b,c,d,e,f,g,h,b-a+(c,d,e,f,g,h)]}
\nr&\times  \frac{1}{2\pi i}  
\int_t \Gamma\biggl[{t+a,t+1+a/2, t+(b,c,d,e,f,g,h), -t+b-a, -t \atop t+a/2 , 
t+1+a -(c,d,e,f,g,h)}\biggr] \,dt,
\end{align}
respectively.

In our limit derivations, we will frequently
make use of the extension of
Stirling's formula to the complex numbers
(see \cite[Section 4.42]{Titch} or
\cite[Section 13.6]{WhitWat}):
\begin{equation}
\label{Str}
\Gamma(a+z)
=\sqrt{2\pi}
z^{a+z-1/2}
e^{-z}
(1+\mbox{O}(1/|z|))
\mbox{ uniformly as }
|z|\to\infty,
\end{equation}
provided that
$-\pi+\delta \leq \mbox{arg}(z) \leq\pi-\delta$ for some fixed $ 
\delta \in (0,\pi)$.

When simplifying expressions involving
gamma functions, we will often use the
recursion and reflection formulas:
\begin{equation}
\label{gammaref}\gg{s+1}=s\gg{s};\qquad
\Gamma(s)\Gamma(1-s)
=\frac{\pi}{\sin \pi s}.
\end{equation}

\section{The function  $M(\vw)$: Coxeter groups, actions, and orbits}
\label{sec3}
 
In this section, we review some results 
concerning the function $M(\vw)$ (cf. \eqref{Mdef}); all of these results may be found in
 \cite{Roy}.

To this end, we first recall some notions from Coxeter group theory.  The Dynkin diagram of the Coxeter group $W(E_7)$ is given by
the graph with vertices labeled $1,2,3,4,5,6,3'$, where
$i,j \in \{1,2,3,4,5,6\}$ are connected by an edge if and only if
$|i-j|=1$, and $3'$ is connected to $4$ only. The presentation
of $W(E_7)$ is given by
\begin{equation}W(E_7)=\langle s_1,s_2,s_3,s_4,s_5,s_6,s_{3'}:(s_i s_j)^{m_{ij}}=1 \rangle,\label{wd7def}\end{equation}
where $m_{ii}=1$ for all $i$; and for $i$ and $j$ distinct, $m_{ij}=3$
if $i$ and $j$ are connected by an edge, and $m_{ij}=2$ otherwise.
It is well-known that the order of $W(E_7)$ is
$72\times 8! = 2903040$ (see \cite[Section $2.11$]{Hum}).

The Dynkin diagram of the Coxeter group $W(E_6)$ is given
by the same graph as the Dynkin diagram of the Coxeter group $W(E_7)$,
except that we omit vertex $1$ (and any connection to it). The presentation
of $W(E_6)$ is the same as the presentation
of $W(E_7)$, except that we omit the generator $s_1$ (and any relations involving $s_1$).
The order of $W(E_6)$ is
$72\times 6! = 51840$ (see \cite[Section $2.11$]{Hum}).

We consider the matrix group $GL(8,\mathbb{C})$ 
acting on the complex vector space
$\mathbb\C^8$ from the left.  If $\sigma \in S_8$, 
where $S_n$ is the symmetric group on $n$ elements,
we will
identify $\sigma$ with the element of $GL(8,\mathbb{C})$ that permutes
the rows of the identity matrix $I_8 \in GL(8,\mathbb{C})$
according to the permutation $\sigma$. Equivalently,
we identify $\sigma$ with the element of $GL(8,\mathbb{C})$ that permutes
the standard basis 
$\{\vec{e}_1,\vec{e}_2,\ldots,\vec{e}_8\}$ 
of the complex vector space
$\mathbb{C}^8$ according to the permutation $\sigma$.
For example,
\begin{equation*}
(123)=
\left( \begin{array}{cccccccc}
0 & 0 & 1 & 0 & 0 & 0 & 0 & 0\\
1 & 0 & 0 & 0 & 0 & 0 & 0 & 0\\ 
0 & 1 & 0 & 0 & 0 & 0 & 0 & 0\\ 
0 & 0 & 0 & 1 & 0 & 0 & 0 & 0\\ 
0 & 0 & 0 & 0 & 1 & 0 & 0 & 0\\ 
0 & 0 & 0 & 0 & 0 & 1 & 0 & 0\\ 
0 & 0 & 0 & 0 & 0 & 0 & 1 & 0\\  
0 & 0 & 0 & 0 & 0 & 0 & 0 & 1     
\end{array} \right).
\end{equation*}

We also define $X,Y \in GL(8,\mathbb{C})$ to be the matrices
\begin{equation}
\label{Xdef}
X=
\left( \begin{array}{cccccccc}
1/2 & 1/2 & -1/2 & -1/2 & -1/2 & 1/2 & 1/2 & 1/2\\
0 & 1 & 0 & 0 & 0 & 0 & 0 & 0\\ 
-1/2 & 1/2 & 1/2 & -1/2 & -1/2 & 1/2 & 1/2 & 1/2\\ 
-1/2 & 1/2 & -1/2 & 1/2 & -1/2 & 1/2 & 1/2 & 1/2\\ 
-1/2 & 1/2 & -1/2 & -1/2 & 1/2 & 1/2 & 1/2 & 1/2\\ 
0 & 0 & 0 & 0 & 0 & 1 & 0 & 0\\ 
0 & 0 & 0 & 0 & 0 & 0 & 1 & 0\\  
0 & 0 & 0 & 0 & 0 & 0 & 0 & 1         
\end{array} \right)
\end{equation}
and
\begin{equation}
\label{Ydef}
Y=
\left( \begin{array}{cccccccc}
-1 & 2 & 0 & 0 & 0 & 0 & 0 & 0\\
-1 & 1 & 1 & 0 & 0 & 0 & 0 & 0\\ 
0 & 1 & 0 & 0 & 0 & 0 & 0 & 0\\ 
-1 & 1 & 0 & 1 & 0 & 0 & 0 & 0\\ 
-1 & 1 & 0 & 0 & 1 & 0 & 0 & 0\\ 
-1 & 1 & 0 & 0 & 0 & 1 & 0 & 0\\ 
-1 & 1 & 0 & 0 & 0 & 0 & 1 & 0\\  
-1 & 1 & 0 & 0 & 0 & 0 & 0 & 1         
\end{array} \right).
\end{equation}

We note that, if
$\vec{w}=(a,b,c,d,e,f,g,h)^T \in W$ 
(see (\ref{Wdef}) for the definition of the affine hyperplane $W$)
then,
by virtue of the relation $2+3a=b+c+d+e+f+g+h$, 
we have \begin{equation}\label{Xaction}X\vec{w}
=(1+2a-c-d-e,b,1+a-d-e,1+a-c-e,1+a-c-d,f,g,h)^T\end{equation}
and
\begin{equation}\label{Yaction}Y\vec{w}
=(2b-a,b+c-a,b,b+d-a,b+e-a,b+f-a,b+g-a,b+h-a)^T.\end{equation}

We define the subgroup
$H$ of $GL(8,\mathbb{C})$ by
\begin{equation}
\label{HgroupDef}
H 
=\langle
(23),(34),(45),(56),(67),(78),X,Y\rangle. 
\end{equation}

It is shown in \cite{Roy}
that $H$ is isomorphic to $W(E_7)$. 
The Coxeter generators of $H$, cf. \eqref{wd7def}, are given by
\begin{equation}
\label{CoxGenH}
s_1=Y(23),s_2=(34),s_3=(45),s_4=(56),
s_5=(67),s_6=(78),s_{3'}=X.
\end{equation}

It is well-known that the Coxeter group
$W(E_7)$ has a center consisting
of two elements (see \cite[Sections 3.20 and 6.4]{Hum}).
We denote by $Z$ the unique nonidentity element
in the center of $H$. It is shown in \cite{Roy}
that, if
$\vec{w}=(a,b,c,d,e,f,g,h)^T$ is in $W$
(see (\ref{Wdef})), then
\begin{equation}Z\vec{w}
=(1-a,1-b,1-c,1-d,1-e,1-f,1-g,1-h)^T\in W.\label{Zaction}\end{equation}

Let $G$ be the subgroup of $H$ defined by
\begin{equation}
\label{GgroupDef}
G
=\langle
s_2,s_3,s_4,s_5,s_6,s_{3'}
\rangle. 
\end{equation}
Then $G$ is isomorphic to $W(E_6)$.
It is shown in \cite{Roy} that $G$
is an invariance group for the $M$ function. 
That is,$$ M(\alpha\vw)=M(\vw)$$for each $\alpha\in G$, $\vw\in W$. 
Since the order of $G$ is $51840$, we have
$51840$ invariances, or two-term relations (including the trivial one),
satisfied by the $M$ function.

For $\tau\in G\backslash H$, we will write $M_\tau$ to denote the function defined by 
$M_{\tau}(\vw)=M(\alpha\vw)$, where $\alpha\in H$ is any element of the coset $\tau$.

There are 
$|H|/|G|=2903040/51840=56$ elements of the coset space $G\backslash H$. We assign to these elements the symbols $\pm v(i,j)\ (0 \leq i < j \leq 7)$ in the following way:
Let $\vec{w}=(a,b,c,d,e,f,g,h)^T \in W$, and put 
\begin{equation}
x_0=b,x_1=h,x_2=g,x_3=f,x_4=e,x_5=d,x_6=c,x_7=a.\label{x_idefs}
\end{equation}
It follows from the results in \cite{Roy} that:

\begena
\item  For $\alpha,\beta\in H$,  $\alpha\vec{w}$ and
$\beta\vec{w}$ have the same second entry  if and only if $\alpha$ and $\beta$ are in the same right coset of
$G$ in $H$; and 

\item  For any $\alpha\in H$, the second entry of $\alpha\vec{w}$  equals either $x_i+x_j-x_7$ or $1+x_7-x_i-x_j$ for some integers $i$ and $j$ with  $0 \leq i \ne j \leq 7$.
\enden 

We may therefore make the following.

\begin{Definition}\label{vij}We let 
$v(i,j)$ denote the right coset with the property that,
if $\alpha \in v(i,j)$, then
$\alpha \vw$ has the second entry
$x_i+x_j-x_7$.  Also, we let 
$-v(i,j)$ denote the right coset with the property that,
if $\beta \in -v(i,j)$, then
$\beta \vw$ has the second entry
$1+x_7-x_i-x_j$.   \end{Definition}

\begin{Remark}
\label{Mlabels}
(a)
Our indexing of the 56 cosets $\pm v(i,j)$ ($0\le i<j\le 7$) is different from
the one used in \cite{Roy}. We use
a different indexing in this paper in order
to obtain a certain canonical form for the results in Section 5.

(b) If we consider the unique nonidentity matrix $Z$
in the center of $H$, we have that, for
any $0\leq i < j \leq 7$, 
the action of $Z$ by right multiplication 
interchanges the cosets $v(i,j)$ and $-v(i,j)$.   This is clear from the above definitions of $v(i,j)$ and $-v(i,j)$, and from \eqref{Zaction}. 
\end{Remark}

 The natural action of $H$ on  $G \backslash H$ may now be described, in terms of the generators $s_1, s_2,\ldots,s_6,s_{3'}$  of $H$ and the indices $\pm v(i,j)$, as follows.
   
\begin{Proposition} \label{M_action_labels}   

\begena 
 
\item[\rm(a)] For $0\le k\le 5$, denote by $\rho_k$ the transposition $(k+1,k+2)$.  Then
$$s_{6-k}(\pm v(i,j))=  \pm v(\rho_k(i),\rho_k(j))  . $$
 \item[\rm(b)]  Define 
$K_0 = \{0, 1, 2, 3\}$ and $K_1 = \{4, 5, 6, 7\}$.  Then
$$s_{3'}(\pm v(i, j)) = \begin{cases} \mp v(k, l)&\text{if $\{i, j, k, l\} = K_p$ for 
some $p \in \{0, 1\}$},\\   \pm v(i, j) &\text{if } not.\end{cases}$$

 \enden\end{Proposition}
 
 \begin{proof}This may be checked directly. \end{proof}
 
 For example, let us compute $s_{3'} M_{v(2,5)}$  in two equivalent ways.  On the one hand we note, using \eqref{x_idefs}, that $v(2,5)\vec{w}$ has second entry $d+g-a$; using \eqref{CoxGenH} and \eqref{Xaction}, we see that $ s_{3'}(v(2,5))\vec{w}$ also has second entry $d+g-a$, so that  $s_{3'} M_{v(2,5)}=M_{v(2,5)}$.  On the other hand, Proposition 
\ref{M_action_labels}(b) gives the same result, 
since  the integers 2 and 5 do not both belong to the same $K_i$, for $i\in\{0,1\}$.

A metric on the coset space $G\backslash H$ may be defined as follows.  First, to each of our 56 cosets $\pm v(i,j)$, we associate an integer octuple, denoted by the same symbol, as follows:
\begin{equation}
\label{eq:vijdef}
\pm v(i,j)=\pm \biggl(4(\vec{e}_{i+1}+\vec{e}_{j+1})
-\sum_{k=1}^8 \vec{e}_k\biggr).
\end{equation}
For example, we find that
$v(0,1)=(3,3,-1,-1,-1,-1,-1,-1)^T$
and
$-v(4,7)=(1,1,1,1,-3,1,1,-3)^T$.  We have

\begin{Definition}\label{DiscDist}

Let $(v_1,\ldots,v_8)^T$ and $(w_1,\ldots,w_8)^T$ be the vectors associated, by (\ref{eq:vijdef}), to elements $v$ and $w$, respectively, of $G\backslash H$.  The {\it discrete distance} $dd(v,w)$ between  $\vec{v}$ and $\vec{w}$ is defined by
\begin{equation*}
dd({v},{w})
=\frac{1}{16} \sum_{k=1}^8 (v_k-w_k)^2. 
\end{equation*}
\end{Definition}

Under the discrete distance, 
$G\backslash H$ becomes a metric space.
It is shown in \cite{Roy} that the possible discrete distances in
$G\backslash H$ are $0,2,4,$ and $6$; and that the action of $H$ 
by right multiplication
on $G\backslash H$ preserves the discrete distance.

For a set $\Omega$ and an integer $n\geq1$,
we let $(\Omega)^{(n)}$ denote the
set of $n$-element subsets of $\Omega$.
Suppose $\{a,b,c\} \in (G\backslash H)^{(3)}$ has
the unordered multiset of distances
$\{dd(a,b),dd(a,c),dd(b,c)\}$.
If the distances in weakly increasing order are
$\{x,y,z\}$, we will say that the triple
$\{a,b,c\}$ is of
{\it Euclidean type} $xyz$. It is shown in \cite{Roy}
that we have five possible Euclidean types, given by
$222,224,244,444,246$.

The action of $H$ by right multiplication
on $(G\backslash H)^{(3)}$ has five orbits. These five orbits
correspond to the above Euclidean types
$222,224,244,444,246$. Furthermore, it is shown in
\cite{Roy} that, if $\tau_1,\tau_2,\tau_3$
 are distinct elements of
$G\backslash H$, then there is a three-term relation
involving the functions
$\MM_{\tau_1} ,\MM_{\tau_2}$, and $\MM_{\tau_3}$.  The coefficients of such a relation are rational combinations of gamma and sine functions, whose arguments are polynomials in $a,b,c,d,e,f,g,h$.
Because of the five orbits just described, we have 
five different types of three-term relations for the
$M$ function. A three-term relation of each of these
five types is found in \cite{Roy}, and any other three-term relation can
be obtained from one of these five relations through a change of
variable of the form
$\vw \mapsto \mu\vw$, for some
$\mu \in H$, applied to all terms and coefficients of the
original three-term relation. In total, we have
$\binom{56}{3}=27720$ three-term relations
satisfied by the $M$ function.  

\section{The functions  $\JJ(\vec{x})$ and 
$\LL(\vec{x})$: Coxeter groups, actions, and orbits}\label{sec4}
In this section, we review the results 
concerning the function $\JJ(\vec{x})$ (cf. \eqref{Jdef})
obtained in \cite{FGS,GMS},   
the results concerning the function
$\LL(\vec{x})$ (cf. \eqref{Ldef}) obtained in \cite{GMS,M1},  
and the results  of \cite{GMS} that relate $\LL$ functions to $\JJ$ functions.

As we did in Section 3, we first assemble some results from Coxeter group theory.  The Dynkin diagram of the Coxeter group $W(D_n)$ is given by
the graph with vertices labeled $1’,1,2,\ldots,n-1$, where
$i,j \in \{1,2,\ldots,n-1\}$ are connected by an edge if and only if
$|i-j|=1$, and $1’$ is connected to $2$ only. The presentation
of $W(D_n)$ is given by
\begin{equation}W(D_n)=\langle a_{1’},a_1,a_2,\ldots,a_{n-1}:(a_i a_j)^{m_{ij}}=1 \rangle,\label{wdn_gen}\end{equation}
where $m_{ii}=1$ for all $i$; and for $i$ and $j$ distinct, $m_{ij}=3$
if $i$ and $j$ are connected by an edge, and $m_{ij}=2$ otherwise.
The order of $W(D_n)$ is
$2^{n-1}n!$ (see \cite[Section $2.11$]{Hum}).

Similarly to Section \ref{sec3}, we consider the matrix group $GL(7,\mathbb{C})$ 
acting on 
the complex vector space
$\mathbb\C^7$ from the left.  If $\sigma \in S_7$, we will
identify $\sigma$ with the element of $GL(7,\mathbb{C})$  that permutes
the rows of the identity matrix $I_7 \in GL(7,\mathbb{C})$
according to the permutation $\sigma$. As an example,
\begin{equation*}
(123)=
\left( \begin{array}{ccccccc}
0 & 0 & 1 & 0 & 0 & 0 & 0 \\
1 & 0 & 0 & 0 & 0 & 0 & 0 \\ 
0 & 1 & 0 & 0 & 0 & 0 & 0 \\ 
0 & 0 & 0 & 1 & 0 & 0 & 0 \\ 
0 & 0 & 0 & 0 & 1 & 0 & 0 \\ 
0 & 0 & 0 & 0 & 0 & 1 & 0 \\ 
0 & 0 & 0 & 0 & 0 & 0 & 1         
\end{array} \right).
\end{equation*}

We also let  $X_1 \in GL(7,\mathbb{C})$ be the matrix
\begin{equation}
\label{310}
X_1=
\left( \begin{array}{ccccccc}
1 & 0 & 0 & 0 & 0 & 0 & 0 \\
0 & 0 & -1 & 0 & 1 & 0 & 0 \\ 
0 & -1& 0 & 0 &1 & 0 & 0 \\ 
0 & 0 & 0 & 1 & 0 & 0 & 0 \\ 
 0 & 0 & 0& 0 &1&0 & 0 \\ 
0 & -1 & -1&0 &1 & 1& 0 \\ 
0 & -1 & -1 & 0 & 1& 0 & 1
\end{array} \right).
\end{equation}

We note that
if $\vec{x}=(A,B,C,D,E,F,G)^T \in V$ 
(see (\ref{vdef}) for the definition of the affine hyperplane $V$)
then,
by virtue of the relation $E+F+G-A-B-C-D=1$, we have
\begin{equation}X_1\vec{x}
=(A,E-C,E-B,D,E,1+A+D-G,1+A+D-F)^T.\label{A1def}\end{equation}

We consider the subgroup $H_1$ of $GL(7,\mathbb{C})$ defined by:
\begin{equation}
H_1=\langle (12),(23),(34),(56),(67),X_1 \rangle.\label{H1def}
\end{equation}
It is shown in \cite{FGS} and \cite{M1} that $H_1$ is
isomorphic to the Coxeter group $W(D_6)$ of order 23040.
It is also shown that the Coxeter generators 
$a_1,a_2,a_3,a_4,a_5,a_{1’}$ of $H_1$, cf. the presentation
\eqref{wdn_gen}, may be given by
\begin{equation}
\label{CoxGenH1}
a_1=(23),a_2=(34),a_3= X_1,a_4=(56),
a_5=(67),a_{1’}=(14).
\end{equation}(In  earlier references, $a_{1’}=(12)$ was used instead of $a_{1’}=(14)$.  But the latter designation will be more convenient for our present purposes.)

The Coxeter group $W(D_6)$ is well-known to have a 
center consisting of two elements 
(see \cite[pp.\ 82 and 132]{Hum}). We denote by $Z_1$ 
the unique nonidentity element in the center
of $H_1$. The element $Z_1$ is called the {\it central involution},
and is computed to be
\begin{equation*}
Z_1=(14)(23)[[(1234)(567)]^2X_1]^4. 
\end{equation*}
We note that, if $\vec{x}=(A,B,C,D,E,F,G)^T \in V$
(see (\ref{vdef})), then
\begin{equation}
Z_1\vec{x}=
(1-A,1-B,1-C,1-D,2-E,2-F,2-G)^T\in V. \label{w0def}
\end{equation}

We next consider the following subgroups of $H_1$:
\begin{equation}
G_{\JJ}=\langle (23),(34),(56),(67),X_1 \rangle
\label{gj}\end{equation}
and\begin{equation}
G_{\LL}=\langle (12),(23),(34),(67),(57)X_1(57) \rangle.\label{gl}
\end{equation}
We have the following:

\begena\item
It is shown in \cite{FGS} that $G_{\JJ}$ is an invariance group for 
$J(\vec{x})$, i.e.
\begin{equation*}
\JJ(\alpha\vec{x})=\JJ(\vec{x}) \textrm{ for all } \alpha \in G_{\JJ},
\end{equation*}
and, furthermore, that $G_{\JJ}$ is 
isomorphic to the symmetric group $S_6$ of
order 720. This gives 720 invariances, or 
two-term relations, for the 
function $\JJ(\vec{x})$. 

 \item
Similarly, it is shown in \cite{M1} that $G_{\LL}$ is an invariance group for
$\LL(\vec{x})$, i.e.
\begin{equation*}
\LL(\beta\vec{x})=\LL(\vec{x}) \textrm{ for all } \beta \in G_{\LL},
\end{equation*}
and, furthermore, that $G_{\LL}$ is 
isomorphic to the Coxeter group $W(D_5)$
of order 1920. This gives 1920 invariances, or 
two-term relations, for the 
function $\LL(\vec{x})$.\end{enumerate}

We can express the above invariances of 
$\JJ(\vec{x})$ and $\LL(\vec{x})$ succinctly if we reparameterize 
the $\JJ$ and $\LL$ functions as follows:  Let 
\begin{equation}\label{Ktwiddle}
\widetilde{\JJ}(x_0,x_1,x_2,x_3,x_4,x_5)
= \JJ(A;B,C,D; 
 E,F,G )
\end{equation}
and
\begin{equation}\label{Ltwiddle}
\widetilde{\LL}(x_0,x_1,x_2,x_3,x_4,x_5)
= \LL(A,B,C,D; 
  E;F,G ),
\end{equation}
where
\begin{align}\label{twiddleparams}&A=\frac{1}{2}+x_0+x_1+x_2+x_3+x_4+x_5, \hskip2.25pt
B=\frac{1}{2}+x_0+x_1+x_2-x_3-x_4+x_5,\nr&C=\frac{1}{2}+x_0+x_1+x_2+x_3-x_4-x_5, \hskip2.25pt
D=\frac{1}{2}+x_0+x_1+x_2-x_3+x_4-x_5,\nr&E=1+2x_1+2x_2, \ 
F=1+2x_0+2x_1, \ 
G=1+2x_0+2x_2.\end{align}(This reparameterization of the $\LL$ function is equivalent to the
reparameterization given in \cite[Eq.\ (3.1)]{Wh3}.)
Then the following is readily shown: 

\begena
\item[(a)] The invariance of $\JJ(\vec{x})$ under the group
$G_{\JJ}\cong  S_6$  amounts to the
equivalence of the 720 functions
in the set  
\begin{align*}\{&\widetilde{\JJ}(x_{i_0},x_{i_1},x_{i_2},x_{i_3},x_{i_4},x_{i_5})\colon (i_0,i_1,i_2,i_3,i_4,i_5)  \\&\hbox{is a permutation of }
(0,1,2,3,4,5)\}.\end{align*}

\item[(b)] The invariance of $\LL(\vec{x})$ under the group
$G_{\LL}\cong  W(D_5)$  amounts to the
equivalence of the 1920 functions in the set\begin{align*}\{& \widetilde{\LL}(x_0,\pm x_{i_1},\pm x_{i_2},\pm x_{i_3},\pm x_{i_4},\pm x_{i_5})\colon (i_1,i_2,i_3,i_4,i_5)  \hbox{ is a permutation}\\&\hbox{ of }
(1,2,3,4,5)\hbox{ and the number of negative signs is even}\}.\end{align*} 
\end{enumerate}

The three-term relations for $\JJ(\vec{x})$ and
$\LL(\vec{x})$ are governed by 
the right cosets of their invariance
groups $G_{\JJ}$ and $G_{\LL}$, respectively, in $H_1$, as follows.

We discuss first the three-term relations
for $\JJ(\vec{x})$.
The number of right cosets of $G_{\JJ}$ in $H_1$ is 
${|H_1|}/{|G_{\JJ}|}={23040}/{720}=32$.  Moreover, from the definitions  \eqref{H1def} and \eqref{gj} of $H_1$ and $G_J$, we may directly deduce the following two facts.
\begen 

\item[(1)] The right $G_\JJ$  coset to which a given $\alpha\in H_1$ belongs is uniquely determined by the first coordinate of $\alpha \vx$, for $\vx\in(A,B,C,D,E,F,G)^T \in V$. (This first coordinate is well-defined, i.e. constant on any coset.)  

\item[(2)] If we write $$(A_0,A_1,A_2,A_3)=(A,B,C,D) \hbox{\quad and\quad} (E_0,E_1,E_2,E_3)=(1,E,F,G),$$ then the set of possible first coordinates of $\alpha \vx$ is 
$$\bigl\{1+A_r-E_q \colon  0\le r,q\le 3\bigr\}\cup  \bigl\{E_q-A_r \colon  0\le r,q\le 3\bigr\}.$$
\enden

Using the above observations, we introduce two different, but equivalent,
labelings of the elements of $G_\JJ \backslash H_1$. The first of these labelings is useful because it is compact; the second is more convenient for describing the action of $H_1$ on $G_\JJ \backslash H_1$, cf. Proposition 4.2 below.

\begin{Definition} \label{J_coset_labels}\begen\item[(a)] If $\sigma\in G_\JJ\backslash H_1$ is such that, for any $\alpha \in \sigma$, $\alpha\vx$ has first coordinate $1+A_r-E_q$, then we denote $\sigma$ by $p_{4q+r}$. If $\sigma$ is such that this first coordinate is $E_q -A_r$, then we denote $\sigma$ by $n_{4q+r}$.
\item[(b)]
Of all length-three strings of plus and minus signs, consider those having evenly many minus signs, denoted as follows:
\begin{equation*}  b_0 ={+}{+}{+},\quad  b_1 ={+}{-}{-},\quad b_2 ={-}{+}{-}, \quad b_3 ={-}{-}{+}.\end{equation*}
(If we replace each ``$+$'' with a ``$0$'' and each ``$-$'' with a ``$1$,'' then $b_i$ is an increasing function of $i$.) Then we assign, to the coset $p_{4q+r}$, the length-six string $b_r b_q$, meaning the concatenation of $b_r$ followed by $b_q$. We also assign, to the coset $n_{4q+r}$, the length-six string obtained by reversing all signs in the string for $p_{4q+r}$.\enden
\end{Definition}
For example, if $\alpha\vx$, for $\alpha\in\sigma$, has first coordinate $1+D-F$, then $\sigma = p_{4\times2+3} = p_{11} ={-}{-}{+}{-}{+}{-}    $. If this first coordinate is $G-B$, then $\sigma = n_{4\times3+1} = n_{13}=  {-}{+}{+}{+}{+}{-}$.

The logic behind the above definition — particularly part (b) — is elucidated by the following proposition.

\begin{Proposition}\label{J_action_labels} The natural action of $H_1$ on  $G_\JJ\backslash H_1$ may be described, in terms of the generators   of $H_1$, as follows:

 \begen 
 
\item[\rm(a)] The generator $a_k$  $(1\le k\le 5)$ transposes the $k$th and $(k+1)$st signs in the string for the  coset  $\sigma$, while  
 
 \item[\rm(b)] the generator $a_{1’}$ replaces the first sign with the negation of the second, and the second with the negation of the first. 
 
 \enden\end{Proposition}
 
 \begin{proof}This follows from Definition 5.12, and the ensuing discussion, in \cite{FGS}.  \end{proof}
 
 For example, consider the coset $n_6$ and the generator $a_3$; we can compute $a_3 n_6$  in two equivalent ways.  On the one hand, since $6=4\cdot1+2$,  $n_6$ consists of elements $\alpha\in H_1$ such that $\alpha \vx$ has first coordinate $E_1-A_2=E-C$, by Definition 4.1.  But $a_3=X_1$ takes $E-C$ to $B$, by   \eqref{A1def}.  So $a_3 n_6$ consists of elements $\alpha\in H_1$ such that $\alpha \vx$ has first coordinate $B;$ so by Definition 4.1, $a_3n_6= p_1$.
 
 On the other hand, we have $n_6=    {+}{-} {+}{-} {+}
 {+}$.  According to Proposition 4.2, application of the generator $a_3$ to this string gives the string ${+} {-}{-} {+}{+} {+}$, which equals $ p_1$.  That is (again), $a_3 n_6 =  p_1.$

     We now define a metric on $G_J\backslash H_1$, and use this metric to define the {\it Hamming type} of a triple of elements of this coset space.
\begin{Definition}
\label{hamming}

\begen

\item[\rm(a)]
We define the 
{\it Hamming distance}  $d(\sigma,\tau)$ between elements 
$\sigma,\tau \in G_{\JJ}\backslash H_1$ to be the number of coordinates at which the strings for $\sigma$ and $\tau$ disagree.

\item[\rm(b)] 
Let $S(\JJ^3)$ denote the set of three-element subsets of  $G_{\JJ}\backslash H_1$. The 
{\it Hamming type} of  
$\{\sigma,\tau,\mu\}\in  S(\JJ^3)$ is 
defined to be the three-digit integer
$abc$, where $a$ is the shortest among the Hamming distances
  $d(\sigma,\tau)$,   $d(\sigma,\mu)$, and   $d(\tau,\mu)$; $b$ 
is the next shortest;
and $c$ is the longest. 
\enden
\end{Definition}

We note that $|S(\JJ^3)|=\binom{32}{3}=4960$. It is shown in  \cite[Proposition 6.5]{FGS} 
that each element of $S(\JJ^3)$ has Hamming type $222,224,244,246$, or $444$, 
and that the action of $H_1$ on $S(\JJ^3)$ by right multiplication (elementwise) partitions $S(\JJ^3)$
into five orbits, with each orbit consisting of all elements of  a given Hamming type.

We write $\JJ_\sigma(\vec{x})$ for $\JJ(\alpha\vec{x})$ 
whenever  $\alpha\in H_1$ belongs to the right coset $\sigma\in G_{\JJ}\backslash H_1$. A  three-term relation among the functions
$\JJ_{\sigma_1} , \JJ_{\sigma_2}$, and 
$\JJ_{\sigma_3}$
is called an $abc$ {\it relation}
if  
$\{\sigma_1,\sigma_2,\sigma_3\}$ is of Hamming
type $abc$. Explicit three-term relations for each of the Hamming types $222,224, 244, 246$, and $444$
 are obtained
in Propositions $7.3$, $7.4$, $7.5$,
$7.6$,  and $7.7$, respectively, of \cite{FGS}.  Applying the action of $H_1$ to the $\JJ$ functions and the coefficients of these five relations, we thereby obtain 4960 three-term relations, partitioned by Hamming type into five families.

\begin{Remark}
\label{R315} Let $\widetilde{\JJ}$ be as defined in equations (\ref{Ktwiddle}) and (\ref{twiddleparams}) above.  Then the set of 23040 $\JJ$ functions that relate to each other via the three-term relations just described is equal to the set
\begin{align*}\{&\widetilde{\JJ}(\pm x_{i_0},\pm x_{i_1},\pm x_{i_2},\pm x_{i_3},\pm x_{i_4},\pm x_{i_5})\colon 
 (i_0,i_1,i_2,i_3,i_4,i_5)\hbox{ is a permutation}\nr
&\hbox{of }(0,1,2,3,4,5)\hbox{ and the number of negative signs is even}\}.\end{align*}
Moreover, the set 
$$\{\widetilde{\JJ}( \pm x_0,\pm x_1,\pm x_2,\pm x_3,\pm x_4,\pm x_5)\colon \hbox{the number of negative signs is even}\}$$equals the set$$\{\JJ_\sigma(\vec{x})\colon \sigma\in G_{\JJ}\backslash H_1\}.$$
\end{Remark}
We next review three-term relations for $\LL(\vec{x})$.
The number of right cosets of $G_{\LL}$ in $H_1$ is 
${|H_1|}/{|G_{\LL}|}={23040}/{1920}=12$. The 12 right cosets
of $G_{\LL}$ in $H_1$ are indexed by 
$1,2,\ldots,6,\overline{1},\overline{2},\ldots,\overline6$, where,  for
each $i \in \{1,2,\ldots,6\}$,   $i $ and ${\overline{i}}$ are
interchanged under the action of the central involution $Z_1$. In particular, 
we define
the corresponding $\LL$ functions as follows:
\begin{align}\label{Lcosetsdef}
& \LL_6(\vec{x})=\LL\left(A,B,C,D;   G;F,E \right),\nr&
 \LL_5(\vec{x})=\LL\left(A,B,C,D;  F;E,G \right),\nr&
 \LL_4(\vec{x})=\LL\left(A,B,C,D;  E;F,G \right),\nr &
 \LL_3(\vec{x})=\LL\left(A,1+A-E,1+A-F,1+A-G;  1+A-D;1+A-B,1+A-C \right),\nr
&\LL_2(\vec{x})=\LL\left(A,1+A-E,1+A-F,1+A-G;  1+A-C;1+A-B,1+A-D \right),\nr&
\LL_1(\vec{x})=\LL\left(A,1+A-E,1+A-F,1+A-G;  1+A-B;1+A-C,1+A-D \right),\nr 
&\LL_{\overline6}(\vec{x})=\LL\left(1-A,1-B,1-C,1-D;  2-G;2-F,2-E \right),\nr&  
\LL_{\overline{5}}(\vec{x})=\LL\left(1-A,1-B,1-C,1-D;  2-F;2-E,2-G \right),\nr
&\LL_{\overline{4}}(\vec{x})=\LL\left(1-A,1-B,1-C,1-D;  2-E;2-F,2-G \right),\nr&
 \LL_{\overline{3}}(\vec{x})=\LL\left(1-A,E-A,F-A,G-A;  1+D-A;1+B-A,1+C-A \right),\nr
&\LL_{\overline{2}}(\vec{x})=\LL\left(1-A,E-A,F-A,G-A;  1+C-A;1+B-A,1+D-A \right),\nr
&\LL_{\overline{1}}(\vec{x})=\LL\left(1-A,E-A,F-A,G-A;  1+B-A;1+C-A,1+D-A \right). 
\end{align} 

In Proposition \ref{J_action_labels}, we described the action of $H_1$ on the functions $J_\sigma$ in terms of the string  for $\sigma$ (cf. Definition \ref{J_coset_labels}).  Similarly, we may now  describe the action of $H_1$ on the functions $L_\sigma$ in terms of the indices  $1,2,\ldots,6,\overline1,\overline2,\ldots,\overline6$.

\begin{Proposition} \label{L_action_labels} The natural action of $H_1$ on  $G_\LL\backslash H_1$ may be described, in terms of the  generators   of $H_1$ and the indices of $G_\JJ\backslash H_1$, as follows:

 \begen 
 
\item[\rm(a)] The generator $a_k$  $(1\le i\le 5)$ exchanges the index $k$ with $k+1$, and  $\overline{k}$ with $\overline{k+1}$.    
 
 \item[\rm(b)] The generator $a_{1’}$ exchanges the index $1$ with $\overline2$, and $\overline1$ with $2$.
 
 \enden\end{Proposition}
 
 \begin{proof}This may be checked directly. \end{proof}
 
 For example, consider the function $L_2$ and the generator $a_2$; we can compute $a_2 L_2$  in two equivalent ways.  On the one hand, we have $a_2=(34)$, so that $a_2$ exchanges $C$ and $D$; so from \eqref{Lcosetsdef}, we see that $a_2 L_2=L_3$.  On the other hand, this last equality follows immediately from  Proposition \ref{L_action_labels}(a).

\begin{Remark}
\label{LcosetsRem} 
The indexing of  $\LL$ functions used in \eqref{Lcosetsdef} is
slightly different from the 
labeling used in \cite{GMS,M1}. In particular,
as compared to those references,
we have  interchanged the following indices:
$4\leftrightarrow 6, \overline{4}\leftrightarrow \overline{6},
1\leftrightarrow 3, \overline{1}\leftrightarrow \overline{3}$. This new indexing will allow us to express a certain equivalence of group actions in a particularly simple way, cf.  Theorem \ref{D6action} below.\end{Remark}

We have
\begin{Definition}
\label{coherdef}
A subset of $G_{\LL} \backslash H_1$ is  is called 
{\it $\LL$-coherent}
if no two elements of this subset are interchanged by the action of the central
involution $Z_1$. A subset of $G_{\LL} \backslash H_1$ that is not $\LL$-coherent  is called 
{\it $\LL$-incoherent}.  
\end{Definition}

We note that a subset of $G_{\LL}\backslash H_1$ is   $\LL$-coherent if and only if it does not
contain both elements of the form $i$ and $\overline{i}$, for any
$i \in \{1,2, \ldots, 6\}$.

The group $H_1$ acts by right multiplication (elementwise) on the set $S(\LL^3)$ of 
three-element subsets of
$G_{\LL} \backslash H_1$.
There are $\binom{12}{3}=220$ 
such subsets. It is shown in 
\cite[Proposition $6.6$]{M1} that this group action partitions $S(\LL^3)$ into two orbits — an orbit of length 160,
 consisting of the $\LL$-coherent elements of $S(\LL^3)$, and an  orbit of
length 60, consisting of the $\LL$-incoherent elements.
 
A three-term relation among the functions
$\LL_i(\vec{x}), \LL_j(\vec{x})$, and $\LL_k(\vec{x})$, for $i,j,k\in \{1,2,\ldots,6,\overline1,\overline2,\ldots,\overline6\}$, is said to
be {\it $\LL$-coherent} if  $\{i,j,k\}$ is an $\LL$-coherent set, and {\it $\LL$-incoherent} otherwise.
Explicit examples of $\LL$-coherent and $\LL$-incoherent relations are given in  Propositions $7.2$ and  $7.4$, respectively, of \cite{M1}.
  Applying the action of $H_1$ to the $\LL$ functions and the coefficients of these two relations, we thereby obtain 220 three-term relations, partitioned by coherence into two families.

\begin{Remark}
Let $\widetilde{\LL}$ be as defined in equations (\ref{Ltwiddle}) and (\ref{twiddleparams}) above.  Then the set of 23040 $\LL$ functions that relate to each other via the three-term relations just described is equal to the set
\begin{align*}\{&\widetilde{\LL}(\pm x_{i_0},\pm x_{i_1},\pm x_{i_2},\pm x_{i_3},\pm x_{i_4},\pm x_{i_5})\colon 
 (i_0,i_1,i_2,i_3,i_4,i_5)\hbox{ is a permutation}\nr
&\hbox{of }(0,1,2,3,4,5)\hbox{ and the number of negative signs is even}\}.\end{align*}
Moreover,  the set
$$\{\widetilde{\LL}(\pm(x_i, x_{j(i)}, x_{k(i)}, x_{\ell(i)}, x_{m(i)}, x_{n(i)}))\colon 0\le i\le 5 \},$$where, for a given $i\in\{0,1,2,3,4,5\}$,    $j(i),k(i),\ell(i),m(i)$, and $n(i)$ denote the distinct elements of $\{0,1,2,3,4,5\}\backslash \{i\}$ (in some fixed order), 
and the plus/minus sign applies to the entire sextuple of coordinates,
equals the set$$\{\LL_\sigma(\vec{x})\colon \sigma\in G_{\LL}\backslash H_1\}.$$
\label{R320} 
\end{Remark}

The three-term relations for $J$ functions and for $L$ functions discussed above fit into a larger framework that also includes ``mixed'' three-term relations.  By the latter, we mean relations involving one $\JJ$ and two $\LL$
functions — called $(\JJ,\LL,\LL)$ relations — and relations involving
one $\LL$ and two $\JJ$
functions — called $(\LL,\JJ,\JJ)$ relations.   We conclude this section by recalling some details, developed in \cite{GMS}, of this larger framework. 

We define
\begin{equation*}
T=(G_{\LL}\backslash H_1) \cup (G_{\JJ}\backslash H_1).
\end{equation*}
We have $|T|=12+32=44$. The notions
of Hamming distance and Hamming type on $G_{\JJ}\backslash H_1$, cf. Definition \ref{hamming}(a)(b) above, are extended in \cite{GMS}
to obtain notions of distance and type on $T$ as follows:

\begin{Definition}  
\label{D520} \begen\item[\rm(a)]
If $\sigma_1,\sigma_2\in T$, then the  {\it distance} $d(\sigma_1,\sigma_2)$ between $\sigma_1$ and $\sigma_2$ is defined to be
\begen\item[\rm(i)]  $0$ if $\sigma_1=\sigma_2$.
\item[\rm(ii)] $2$  if:  $\sigma_1,\sigma_2\in G_{\JJ}\backslash H_1$ and the Hamming distance  between $\sigma_1$ and $\sigma_2$ is 2, {or}   (in the case where either $\sigma_1$ or $\sigma_2$ belongs to $G_{\LL}\backslash H_1$) $\sigma_1$ and $\sigma_2$ are not opposite. Here, two cosets in $T$ are said to be {\it opposite}  if some element of the former coset equals the central involution times some element of the latter.  Otherwise,  they are not opposite.

 \item[\rm(iii)]  $4$  if:  $\sigma_1,\sigma_2\in G_{\JJ}\backslash H_1$ and the Hamming distance  between $\sigma_1$ and $\sigma_2$ is 4, {or}   (in the case where either $\sigma_1$ or $\sigma_2$ belongs to $G_{\LL}\backslash H_1$) $\sigma_1$ and $\sigma_2$ are  opposite. 

\item[\rm(iv)] $6$ if  $\sigma_1,\sigma_2\in G_{\JJ}\backslash H_1$ and the Hamming distance  between $\sigma_1$ and $\sigma_2$ is 6.
 \end{enumerate}
\item[\rm(b)] By the {\it type} of  a three-element subset $\{\sigma_1,\sigma_2,\sigma_3\} $ of $T$, we  mean the symbol $abc$, where $a$, $b$, and $c$ are the integers $d(\sigma_1,\sigma_2)$, $d(\sigma_1,\sigma_3)$, and $d(\sigma_2,\sigma_3)$, written in weakly increasing order.  

\end{enumerate}
\end{Definition}
 
We denote  by $P_\sigma$, for $\sigma\in T$,  the $\JJ$ or $\LL$ function associated with $\sigma$ — that is, $P_\sigma(\vec{x})= \LL _\sigma (\vec{x})$ if $\sigma\in G_{\LL}\backslash H_1$, and $P_\sigma(\vec{x})= \JJ_\sigma( \vec{x})$ if $\sigma\in G_{\JJ}\backslash H_1$.  We also write
$$T^{(3)}=\{\hbox{\rm three-element subsets }\mathcal{S}\subset T\}.$$

The following results are proved in \cite[Theorem 5.1]{GMS}.

\begin{Theorem}\label{bigthm}  
\begin{enumerate}
\item[\rm(a)] The action of $H_1$ by right multiplication (elementwise) on $T^{(3)}$ partitions this set  into eighteen orbits. More specifically, there are:

\begin{enumerate}\item[\rm(i)]  two orbits  each of whose elements $\mathcal{S}$ is a set containing three elements of $G_{\LL}\backslash H_1$; 

\item[\rm(ii)]  four  orbits each of whose elements $\mathcal{S}$ is a set  containing one element of $G_{\JJ}\backslash H_1$ and two elements of $G_{\LL}\backslash H_1$; 

\item[\rm(iii)]  seven   orbits each of whose elements $\mathcal{S}$ is a set containing one element of $G_{\LL}\backslash H_1$ and two elements of $G_{\JJ}\backslash H_1$; 

\item[\rm(iv)]   five orbits   each of whose elements $\mathcal{S}$ is a set containing three elements of  $G_{\JJ}\backslash H_1$. 
\end{enumerate}

\item[\rm(b)] Let $\{\sigma_1,\sigma_2,\sigma_3\}\in T^{(3)}.$ Then there is a relation of the form
\begin{equation} \label{relngeneral}
c_1(\vx) P_{\sigma_1} (\vec{x}) + c_2(\vx) P_{\sigma_2} (\vec{x}) + c_3(\vx) P_{\sigma_3} (\vec{x}) = 0.
\end{equation}
Here, each of the coefficients $ c_j(\vx)$ is a   rational combination of sine and gamma functions, whose arguments are all $\Z$-linear combinations of the coordinates $A,B,C,D,E,F,G$ of $\vx$.

\item[\rm(c)] If $\{\sigma_1,\sigma_2,\sigma_3\} $ and $\{\sigma_1’,\sigma_2’,\sigma_3’\}   $ are in the same orbit under the action of $H_1$ on $T^{(3)}$ described above, then a three-term relation among
$P_{\sigma_1}$, $P_{\sigma_2}$, and $P_{\sigma_3}$ can be transformed into one among $P_{\sigma_1’}$, $P_{\sigma_2’}$, and $P_{\sigma_3’}$ by the
application of a single change of variable
$$\vec{x} \mapsto \rho \vec{x} \hspace{1in} (\rho \in H_1)$$
to all elements (including the coefficients) of the first relation.
  \end{enumerate}

 \end{Theorem}

 \begin{Remark} More detail on the above theorem is given in \cite {GMS}.  There,  for each of the
 eighteen orbits of Theorem \ref{bigthm}(a), a three-term relation involving an element $\{\sigma_1,\sigma_2,\sigma_3\} $ of that orbit is given explicitly
 in \cite[Section 6]{GMS}.   Further, the complexity of the  coefficients in a given relation is seen to be related to the distances between the constituent $\JJ/\LL$ functions. 
 
 Generally, such detail will not be required in the present work, though we will write out one particular $\JJ/\LL$ relation explicitly, in   Section 8,  below. \end{Remark}

\section{$\JJ/\LL$ functions from $M$ functions:  an algebraic correspondence}\label{sec5}
\label{AlgebraicCorr}
Our ultimate goal is to establish and describe a correspondence between the set $\{M_\tau\colon \tau \in G\backslash H\}$ of $M$ functions on the one hand, and the set   $ \{\LL_\sigma\colon \sigma \in G_\LL\backslash H_1\}\cup\{\JJ_\mu\colon \mu \in G_\JJ\backslash H_1\}$ on the other. 
(See (\ref{HgroupDef}), (\ref{GgroupDef}), (\ref{H1def}), (\ref{gj}), and (\ref{gl}) for the definitions
of the groups $H,G,H_1,G_J$, and $G_L$, respectively.)
Note that this set of $M$ functions has cardinality 56, while the  set  of $\LL$ and $\JJ$ functions has cardinality $12+32=44$, so our association cannot be one-to-one.  Rather, as we will see, this association is ``two-to-one onto $L$ functions and one-to-one onto $J$ functions,'' corresponding to the fact that $56=12+12+32.$

Actually, we will present two {\it a priori} different correspondences.

\begin{Correspondence}\label{al_corr}
 {An algebraic correspondence:} 
We will show that a certain subgroup $Q$ of $H$, acting in the natural way on the coset space $G\backslash H$, has three orbits; that the action of $Q$ on each of the first two of these orbits is isomorphic to the action of $H_1$ on $G_L\backslash H_1$; and that that the action of $Q$ on the third of these orbits is isomorphic to the action of $H_1$ on $G_J\backslash H_1$.

\end{Correspondence}

\begin{Correspondence}\label{an_corr}
 {An analytic correspondence:} We will show that a certain limit of
any (appropriately normalized) $M_\tau$ is either an $\LL_\sigma$ or a $\JJ_\sigma$, depending on $\tau$.
\end{Correspondence}

The upshot of this section and Section \ref{AnalyticCorr} will be that these two correspondences are essentially the same.  

In this section, we focus on Correspondence \ref{al_corr}.  Key to this  is the intransitive action of a certain subgroup  $Q$ of $H$ on $G\backslash H$.  Specifically, we have the following.

\begin{Proposition}
\label{3orbits} Define a subgroup $Q$ of $H$ by
\begin{equation}
\label{QgroupDef}
Q
=\langle
s_1,s_2,s_3,s_4,s_5,s_{3’}
\rangle. 
\end{equation}Then:
\begena
\item[\rm(a)]    $Q$ is isomorphic to the subgroup $H_1$ of $GL(7,\C)$ given by \eqref{H1def}, via the map $m$ defined, on generators of $Q$, as follows:
\begin{equation}
m(s_k)= a_{6-k}\ (1\le k\le 5),\quad 
m(s_{3’})= a_{1’} .\label{IsoGen}\end{equation}

 (Therefore, $Q\cong W(D_6)$.) 
 
  \item[\rm(b)] There are three orbits $\mathscr{O}_1,\mathscr{O}_2$, and $\mathscr{O}_3$ of the action of
$Q$ on the set of $56$ right cosets of 
$G$ in $H$, namely:\begin{align*}
&\mathscr{O}_1
=\{v(0,j):2\leq j \leq 7\}
\cup
\{-v(1,j):2\leq j \leq 7\},\\
&\mathscr{O}_2
=\{v(1,j):2\leq j \leq 7\}
\cup
\{-v(0,j):2\leq j \leq 7\},\\
&\mathscr{O}_3
=
\{\pm v(i,j):(i,j)=(0,1)  {\rm\ or\ }
2\leq i < j \leq 7\}.
\end{align*}\enden
\end{Proposition}

 \begin{proof}
 The above can be deduced from
\cite[Ex.\ 8.2.11]{Green}. \end{proof}

Note that the orbits
$\mathscr{O}_1,\mathscr{O}_2$, and $\mathscr{O}_3$ are of size
  $12,$ $12$, and $32$, respectively.   This suggests that each of these first two orbits might be identified with our set of 12 $L$ functions, and the third orbit with our set of 32 $J$ functions.  We now provide such an identification.

\begin{Definition}
\label{CosetLabels}
\begena

 \item The elements of  $  \mathscr{O}_1$ are called  {\it blue $L$ cosets}.  We identify $  \mathscr{O}_1$ with $G_\LL\backslash H_1$ via the correspondence $\gamma_1$ given by
\begin{align*}
\gamma_1(v(0,j))=j-1, \quad 
 \gamma_1(-v(1,j)  )= \overline{j-1}\quad (2\le j\le 7).
\end{align*}

 \item The elements of  $  \mathscr{O}_2$ are called  {\it red $L$ cosets}.  We identify $  \mathscr{O}_2$ with $G_\LL\backslash H_1$ via the correspondence  $\gamma_1$ given by
\begin{align*}
\gamma_1(v(1,j))=j-1,\quad
\gamma_1(-v(0,j)) = \overline{j-1}\quad (2\le j\le 7).
\end{align*}

 \item The elements of  $  \mathscr{O}_3$ are called  {\it $J$ cosets}.   We identify $  \mathscr{O}_3$ with $G_\JJ\backslash H_1$ via the correspondence $\gamma_1$ given by
\begin{align*}
\gamma_1(v(0,1)) =\hbox{${+}{+}{+}{+}{+}{+}$} , \quad
\gamma_1(-v(0,1))= {-}{-}{-} {-}{-}{-} , 
\end{align*}
and, for $2\le i<j\le7$:
 \begin{align*} \gamma_1(v(i,j))= &  \hbox{ the string of six plus/minus signs with plus signs in }\\&\hbox{ positions  $i-1$ and $j-1$, and minus signs elsewhere},\\
 \gamma_1(-v(i,j)) = &\hbox{ the negation of }\gamma_1(v(i,j)).
\end{align*}
 
\enden
\end{Definition}Note that   $\gamma_1$ defines a map from  $\mathscr{O}_1\cup \mathscr{O}_2\cup  \mathscr{O}_3=G\backslash H$  to $T=(G_\LL\backslash H_1)\cup (G_\JJ\backslash H_1)$.

It will often be convenient to refer to $\gamma_1(\tau)$, for $\tau\in G\backslash H$, as the {\it label} of $\tau$.  So the above $L$ cosets are identified by their colors and their labels; the $J$ cosets, by their labels alone.

The following theorem shows that the labeling on any one of the three orbits $\mathscr{O}_i$ provides an isomorphism of group actions, with the groups in question being isomorphic to $W(D_6)$.  The theorem also justifies the above  ``blue $L$,'' ``red $L$,'' and ``$J$'' coset terminology.

\begin{Theorem}
\label{D6action}\begen\item[\rm(a)] 
The action of $Q$ on either the set $\mathscr{O}_1$ of blue $L$ cosets or  the set $\mathscr{O}_2$ of red $L$ 
cosets is isomorphic, via the isomorphism $m$ of Proposition \ref{3orbits}{(a)} and the correspondence $\gamma_1$ of Definition \ref{CosetLabels}(a) or (b),
to the action of $H_1$ on the set of 12 cosets of the $\LL$ function.

\item[\rm(b)] 
The   action of $Q$ on the set $\mathscr{O}_3$ of $\JJ$ cosets is equivalent, via the isomorphism $m$ of Proposition \ref{3orbits}{(a)} and the correspondence $\gamma_1$ of Definition \ref{CosetLabels}{(c)},
to the action of $H_1$ on the set of 32 cosets of the
$\JJ$ function.\enden\noindent
In other words, for   any $\tau\in G\backslash H$ and $\alpha\in Q$, we have$$\gamma_1(\alpha\tau)=m(\alpha)\gamma_1(\tau).$$
\end{Theorem}

\begin{proof}
We will prove that the actions are isomorphic by considering the actions of
individual Coxeter generators.

We first consider the case of a generator $a_k\in H_1$, where $k\in \{1,2,3,4,5\}$.
According to Propositions \ref{J_action_labels}(a) and \ref{L_action_labels}(a), this generator acts as the signed permutation $(k,k+1)$ both on the 12 indices $1,2,\ldots,6,\overline{1},\overline2,\ldots,\overline6$ of the
$L$ functions, and also on the 32 strings of plus/minus signs corresponding to the $J$ functions. (In the case of $J$ functions, this means that $a_k$ transposes the $k$th and $(k+1)$st symbols in the string.) On the other
hand, by Proposition \ref{M_action_labels}(a), the generator $s_{6-k}$ corresponding to $a_k$ acts on the 56 elements $\pm v(i,j)$ $(0\le i<j \le 7)$ via
the permutation $(k+1, k+2)$ of the indices $i$ and $j$. A routine case-by-case check
using Definition \ref{CosetLabels} shows that the actions on the three orbits are all
compatible in this case.

It remains to show that the actions of $a_{1’}\in H_1$ and $s_{3’}\in Q$ are compatible.  We will we will do so by showing that the actions of these two generators are compatible on
the element $v(i, j)$; the action on $-v(i, j)$ follows by an analogous 
argument, after negating all the labels and exchanging the roles of blue and
red.

We recall, from Propositions \ref{J_action_labels}(b) and \ref{L_action_labels}(b), that $a_{1’}$ acts on both $\LL$ and $\JJ$ functions  as the signed 
permutation $(1 \overline{2})$, which transposes $1$ with $\overline{2}$, $\overline{1}$ with $2$,
and leaves the other symbols fixed. (In the case of $J$ functions, this means that $a_{1’}$ transposes, and negates, the first and second symbols in the string.)
On the other hand we recall, cf. Proposition \ref{M_action_labels}(b),  that
$s_{3’}(\pm v(i, j))$ equals  $\mp v(k, l)$ if $i, j, k, l$ all belong to  $K_0 = \{0, 1, 2, 3\}$ or all belong to  $K_1 = \{4, 5, 6, 7\}$, and equals $  \pm v(i, j)$ otherwise.

There are six cases to consider, depending on whether each of the indices $i$
and $j$ comes from the set $\{0, 1\}$, the set $\{2, 3\}$, or the set
$\{4, 5, 6, 7\}$. We assume that $i < j$ throughout. Cases 1, 2, 3 and 6
deal with the orbit $\mathcal{O}_3$, and Cases 4 and 5 deal with the
other two orbits.

Case 1: $i, j \in \{0, 1\}$.
In this case, $v(i, j) = v(0, 1)$, whose label is ${+}{+}{+}{+}{+}{+}$,
and $s_{3’}(v(0, 1)) = -v(2, 3)$, whose label is ${-}{-}{+}{+}{+}{+}$. The
action of $a_{1’}$ is to exchange the leftmost two symbols and change the
sign of both of them, which produces the same effect.

Case 2: $i, j \in \{2, 3\}$.
This is similar to Case 1, except that 
$v(i, j) = v(2, 3)$, whose label is ${+}{+}{-}{-}{-}{-}$,
and $s_{3’}(v(2, 3)) = -v(0, 1)$, whose label is ${-}{-}{-}{-}{-}{-}$.

Case 3: $i, j \in \{4,5,6,7\}$.
Define $k, l \in \{4,5,6,7\}$ so that $k < l$ and $\{i,j,k,l\} = \{4,5,6,7\}$.
In this case, the label of $v(i, j)$ has plus signs in positions $i-1$ and 
$j-1$, and minus signs in positions $1$, $2$, $k-1$ and $l-1$.
We have $s_{3’}(v(i, j)) = -v(k, l)$, whose label has minus signs in positions
$k-1$ and $l-1$, and plus signs in positions $1$, $2$, $i-1$ and $j-1$. The
action of $a_{1’}$ changes the two leftmost plus signs to minus signs,
producing the same effect as $s_{3’}$.

Case 4: $i \in \{0, 1\}$ and $j \in \{2, 3\}$.
Define $k$ and $l$ such that $\{0, 1\} = \{i, k\}$ and $\{2, 3\} = \{j, l\}$.
In this case, $v(i, j)$ has label $j-1$, and is blue if $i = 0$ and red if
$i = 1$. We have $s_{3’}(v(i, j)) = -v(k, l)$, whose label is $\overline{l-1}$,
and whose color is the same as that of $v(i, j)$ because $k \ne i$.
Since $\{j-1, l-1\} = \{1, 2\}$, the effect of $a_{1’}$ is to exchange
$j-1$ and $\overline{l-1}$, which produces the same effect.

Case 5: $i \in \{0, 1\}$ and $j \in \{4,5,6,7\}$.
In this case, $\{i, j\}$ is not a subset of either $K_0$ or $K_1$, so $s_{3’}$
acts as the identity. The label of $v(i, j)$ is $j-1$, which is different from
$1$ or $2$, so $a_{1’}$ also acts as the identity, as required.

Case 6: $i \in \{2, 3\}$ and $j \in \{4,5,6,7\}$.
In this case, the label of $v(i, j)$ has plus signs in positions $i-1$ and 
$j-1$, and minus signs elsewhere. As in Case 5, $\{i, j\}$ is not a subset
of either $K_0$ or $K_1$, so $s_{3’}$ acts as the identity. Since one of the
leftmost two positions is a plus and the other is a minus, the action of
$a_{1’}$ is also the identity, as required.
\end{proof} 

 \section{The map $\gamma_1$ and distances between cosets}\label{DistCorr}
 
 Here we show that the above map $$\gamma_1\colon \mathscr{O}_1\cup \mathscr{O}_2\cup  \mathscr{O}_3=G\backslash H\rightarrow T=(G_\LL\backslash H_1)\cup (G_\JJ\backslash H_1)$$
preserves distance, with the exception that the distance between a blue $L$ coset and a red $L$ coset is compressed under this correspondence.
(See, again, (\ref{HgroupDef}), (\ref{GgroupDef}), (\ref{H1def}), (\ref{gj}), and (\ref{gl}) for the definitions
of the groups $H,G,H_1,G_J$, and $G_L$, respectively.)
 
 We begin by examining the discrete distance $dd(v,w)$ (cf. Definition \ref{DiscDist}) more closely.

\begin{Lemma}
\label{LemmaA} 
Let $v  = \pm v(i, j)$ and $w  = \pm v(k, l)$ be (possibly equal) 
cosets of $G$ in $H$, where the signs are chosen independently. Then 
$dd(v , w )$ is given by
\begen
  \item[\rm(i)] $0$ if and only if $v  = w $;
\item[\rm(ii)] $2$ if and only if either 
\begen
\item[\rm(a)] $|\{i, j\} \cap \{k, l\}| = 1$ and $v $ and $w $ have the same sign, or
\item[\rm(b)] $|\{i, j\} \cap \{k, l\}| = 0$ and $v $ and $w $ have opposite signs;
\enden
\item[\rm(iii)] $4$ if and only if either
\begen
\item[\rm(a)] $|\{i, j\} \cap \{k, l\}| = 0$ and $v $ and $w $ have the same sign, or
\item[\rm(b)] $|\{i, j\} \cap \{k, l\}| = 1$ and $v $ and $w $ have opposite signs;
\enden
\item[\rm(iv)] $6$ if and only if $v  = -w $.
\enden
\end{Lemma}

\begin{proof}
This is a restatement of \cite[Lemma $9.1.3$]{Green}, and is proved using a case
by case check.
\end{proof}

\begin{Lemma}
\label{LemmaB} 
Let $v $ and $w $ be as in Lemma \ref{LemmaA}.
\begen
\item[\rm(i)] We have $dd(v , w ) = dd(-v , -w )$; in other words, 
negation on ${\Bbb R}^8$ induces an automorphism of discrete distance.  
\item[\rm(ii)] We have $dd(v , w ) + dd(v , -w ) = 6$.
\enden
\end{Lemma}

\begin{proof}
Part (i) follows from the definitions and the observation 
that negation is an isometry of ${\Bbb R}^8$ with the Euclidean metric.  

Part (ii) follows from \cite[Lemma $9.3.3$]{Green}, which applies to this situation
because of \cite[Lemma $9.3.2$]{Green}.
\end{proof}

\begin{Proposition}
\label{metric}
\begen
\item[\rm(a)] The discrete distance between $J$ cosets $v $ and $w$ equals the number of places (out of six) 
at which the strings $\gamma_1(v)$ and $\gamma_1(w)$ disagree.

\item[\rm(b)] If $v $ and $w$ are $L$ cosets of different colors but the same label, then $dd(v,w)=2$.

\item[\rm(c)] The discrete distance between two distinct $L$ cosets $v $ and $w$ of the same color is 2, unless $v$ and $w$ are opposite, meaning $\{\gamma_1(v),\gamma_1(w)\}=\{i,\overline{i}\}$ for some $1\le i\le 6$.  In the latter case, we have $dd(v,w)=4$.

\item[\rm(d)] If $v $ and $w$ are $L$ cosets of different colors  and different labels, then $dd(v,w)=4$, unless $v $ and $w$ form an opposite pair, in which case $dd(v,w)=6$.

\item[\rm(e)]  Let $v$ be an ``unbarred'' $L$ coset of either color, meaning $v=v(i,j)$ for $i\in\{0,1\}$ and $2\le j \le 7$  (so that the label $\gamma_1(v)$  has no bar).  If $w$ is a $J$  coset,  then $dd(v,w)=2$  if $\gamma_1(w)$ has a plus sign in position $i$, and $dd(v,w)=4$ otherwise.

\item[\rm(f)] Let $v$ be a ``barred'' $L$ coset of either color, meaning $v=-v(i,j)$ for $i\in\{0,1\}$ and $2\le j \le 7$  (so that the label $\gamma_1(v)$ has a bar).  If $w$ is a $J$  coset,  then $dd(v,w)=2$  if $\gamma_1(w)$ has a minus sign in position $i$, and $dd(v,w)=4$ otherwise.
\enden
\end{Proposition}

\begin{proof}
As in Lemma \ref{LemmaA}, we denote the cosets by $v  = \pm v(i, j)$ and 
$w  = \pm v(k, l)$, where the signs are chosen independently.

For the proof of (a), we will show that the discrete distance restricts to 
Hamming distance on $J$ cosets.
There are two cases to consider, depending on whether
or not one of the cosets $\pm v(0, 1)$ appears. If one of these two cosets 
appears, Lemma \ref{LemmaB}(i) shows that we may assume that $v  = v(0, 1)$.
Since $w $ is also a $J$ coset, we must have either $w  = v(0, 1)$, or
$w  = -v(0, 1)$, or $w  = v(i, j)$ with $2 \leq i < j \leq 7$, or 
$w  = -v(i, j)$ with $2 \leq i < j \leq 7$. Lemma \ref{LemmaA} shows that, in each of
these subcases, $dd(v , w )$ is given by $0$, $6$, $4$, and $2$ respectively.
Definition \ref{CosetLabels}(c) shows that $\gamma_1(w)$, for these same subcases, has a total of $0$, $6$, $4$, and
$2$ minus signs respectively. This completes the proof of the first case of (a).

For the other case of (a), we are in the situation where we have
$v  = \pm v(i, j)$ with $2 \leq i < j \leq 7$
and $w  = \pm v(k, l)$ with $2 \leq i < j \leq 7$, 
with signs chosen independently. 
Lemma \ref{LemmaB}(i) allows us to reduce to the case where $v $ has positive sign.

If $w $ also has positive sign, then Lemma \ref{LemmaA} shows that $dd(v , w )$ is
given by $0$, $2$, or $4$, according as $|\{i, j\} \cap \{k, l\}|$ is equal to
$2$, $1$, or $0$, respectively. Definition \ref{CosetLabels}(c) shows that the Hamming 
distance between   $\gamma_1(v) $ and $\gamma_1(w)$, for these cases, is also equal to $0$, $2$, or
$4$ respectively, which proves the assertion in the case where $w $ has
positive sign.

If $w $ has negative sign, then Lemma \ref{LemmaA} shows that $dd(v , w )$ is
given by $6$, $4$, or $2$, according as $|\{i, j\} \cap \{k, l\}|$ is equal to
$2$, $1$, or $0$, respectively. Definition \ref{CosetLabels}(c) shows that the Hamming 
distance between   $\gamma_1(v) $ and $\gamma_1(w)$ is also equal to $6$, $4$, or
$2$ respectively, which completes the proof of (a).

For (b), it follows from Definition \ref{CosetLabels} that we must have 
$v  = \pm v(i, j)$ and $w  = \pm (h, j)$, where $v $ and $w $ have the
same sign, and where $\{h, i\} = \{0, 1\}$. Lemma \ref{LemmaA} then shows that
$dd(v , w ) = 2$, as required.

For (c), it follows by Lemma \ref{LemmaA} that, whenever $1 \leq i, j \leq 6$,
the cosets $v(0, i+1)$   and $-v(1, j+1)$  
all lie at a discrete distance of $2$ from each other, except when $i = j$, 
in which case the cosets lie at a discrete distance of $4$ from each other.
This proves (c) for blue cosets, and the argument for red cosets is the 
same with trivial changes. 

For (d), it follows by Lemma \ref{LemmaA} that, whenever $1 \leq i, j \leq 6$, the
coset $v(0, i+1)$ lies at distance $4$ from 
$v(1, j+1)$   if $j \ne i$, lies
at distance $4$ from $-v(0, j+1)$ 
for $j \ne i$, and lies at distance $6$
from $-v(0, i+1)$, in accordance with the statement.
Changing signs and exchanging the roles of $0$ and $1$ proves the analogous
claim for the coset $-v(1, i+1)$ (with color blue and label $\overline{i}$) and completes
the proof of (d).

For (e), the $L$ coset, $v $, is either of the form $v(0, i+1)$   
or $v(1, i+1)$ for $1 \leq i \leq 6$. We will deal with 
the case of $v(0, i+1)$;
the other case follows by exchanging the roles of $0$ and $1$. As in the proof
of (a), there are two subcases, according to whether the $J$ coset, $w $, 
is of the
form $\pm v(0, 1)$, or is of the form $\pm v(k+1, l+1)$ for 
$1 \leq k, l \leq 6$.

In the first subcase, Lemma \ref{LemmaA} shows that $dd(v , v(0, 1))$ is equal to $2$ and
$dd(v , -v(0, 1))$ is equal to $4$. This agrees with the statement, because  
 $\gamma_1(v(0, 1))$ has a plus sign in position $i$, and  $\gamma_1(-v(0, 1))$ does
not. In the second subcase, Lemma \ref{LemmaA} shows that $dd(v , v(k+1, l+1))$ is
equal to $2$ if $i \in \{k, l\}$, and to $4$ otherwise. By Definition \ref{CosetLabels},
  $\gamma_1(v(k+1, l+1))$ has plus signs in positions $k$ and $l$, and 
minus signs elsewhere, which is in accordance with the statement. We also need
to consider $dd(v , -v(k+1, l+1))$, which by Lemma \ref{LemmaB} (ii) and the above is
equal to $4$ if $i \in \{k, l\}$, and to $2$ otherwise. By Definition \ref{CosetLabels},
  $\gamma(-v(k+1, l+1))$ has minus signs in positions $k$ and $l$, and
plus signs elsewhere, which again agrees with the statement and completes
the proof of (e).

Finally, (f) follows from (e) by applying the negation automorphism of
Lemma \ref{LemmaB}(i).
\end{proof}

\begin{Remark}
\label{rem:metric}
If we consider the representations of the elements
of $G\backslash H$ as vectors
(see (\ref{eq:vijdef})),
then we have the following:
The vectors from the blue $\LL$  cosets have their $0$-coordinate bigger than their $1$-coordinate.
The vectors from the red $\LL$  cosets have their $1$-coordinate bigger than their $0$-coordinate.
The vectors from the $\JJ$  cosets have their $0$-coordinate and $1$-coordinate equal.
\end{Remark}

We may now prove the following.

\begin{Proposition} Let $dd(v,w)$ denote discrete distance on $G\backslash H$, as in Definition \ref{DiscDist}, and let $d(\sigma,\tau)$ denote the distance on $T=(G_{\LL}\backslash H_1) \cup (G_{\JJ}\backslash H_1)$, as in Definition \ref{D520}.  Then
$$d(\gamma_1(v),\gamma(w))=\begin{cases}dd(v,w)-2&\text{if $v$ and $w$ are $L$ cosets of opposite colors},\\dd(v,w)&\text{otherwise}.\end{cases}$$\end{Proposition}

\begin{proof} This is readily checked by comparing Definition \ref{D520}, on a case-by-case basis, to Proposition \ref{metric}.\end{proof}

 \section{
 $J$ and $L$ functions as limits of $M$ functions}
 \label{AnalyticCorr}
We now investigate an analytic means of associating functions  $J(\alpha \vec{x})$ or $ L(\alpha \vec{x}) $ ($\alpha\in H_1$, $\vec{x}\in V$)  to functions $M(\beta \vec{w})$ ($\beta\in H$, $\vec{w}\in W$), cf. Correspondence 5.2 above.  (Here, again, the functions $J$, $L$, and $M$ are defined in 
\eqref{Jdef}, \eqref{Ldef}, and \eqref{Mdef} respectively; the groups $H_1$ and $H$ are defined in \eqref{HgroupDef} and \eqref{H1def} respectively; and the spaces $V$ and $W$ are defined in \eqref{vdef} and \eqref{Wdef} respectively.)

We begin with

\begin{Lemma} 
\label{L610}
Define
\begin{equation}(g)_a=\frac{\gg{g+a}}{\gg{g}}\label{gsuba}\end{equation}
for complex numbers $g$ and $a$.   
Then
$$\lim_{\im{g}\to\pm\infty} \frac{(g+x)_{y}}{(g)_y}=1$$
for any complex numbers $x$ and $y$ {\rm (}not depending on $g${\rm)}.
\end{Lemma}

\begin{proof}  
By (\ref{Str}), we have
\begin{align*}&
\frac{(g+x)_{y}}{(g)_y}=\frac{\gg{g+x+y}\gg{g}}{\gg{g+x}\gg{g+y}}\sim \frac{(\sqrt{2\pi}g^{g+x+y-1/2}e^{-g})\cdot (\sqrt{2\pi}g^{g-1/2}e^{-g})}{(\sqrt{2\pi}g^{g+x-1/2}e^{-g})\cdot (\sqrt{2\pi}g^{g+y-1/2}e^{-g})}\sim 1
\end{align*}as $\im{g}\to\pm\infty$.
\end{proof}

Central to the results of this section will be the fact that a (suitably normalized) function $M(\vw)$ becomes an $L$ function or a $J$ function respectively, in the limit as the parameter $b$ — respectively, $g$ — tends to infinity along the imaginary axis.  Specifically, with the help of the above lemma, we deduce the following.

\begin{Proposition}  
\label{limitsareLorJ} 
Let $a,b,c,d,e,f,g \in \mathbb{C}$, and define \begin{equation}\label{hpreln}h=2+3a-b-c-d-e-f-g,\end{equation} so that
$2+3a=b+c+d+e+f+g+h$.  
Then
\begin{align}
\lim_{{\rm Im}{(b)}\to\pm\infty}
\Gamma&\bigg[{1+a-h,b-a+(c,d,e,f,g)
\atop
b-a}\bigg]
\MM(a; b; c, d, e, f, g, h)\nr
= \frac{\pi}{2}    \LL&\biggl[{e,f,g,1+a-c-d;\atop e+f+g-a; 1+a-c,1+a-d}\biggr]\label{MlimitL}
\end{align}and
\begin{align}
\lim_{{\rm Im}{(g)}\to \pm\infty} \Gamma&[1-b,1+a-(g,h),b-a+(g,h)]
\MM(a; b; c, d, e, f, g, h)\nr
= \frac{\pi}{2 }   \JJ&\biggl[{b;c,d,1+a-e-f;\atop b+c+d-a,1+a-e,1+a-f}\biggr].\label{MlimitJ}
\end{align}

\end{Proposition}

\begin{proof} We first prove \eqref{MlimitL}.  By the Barnes integral representation (\ref{MBarInt}) of the $\MM$ function, 
an interchange of a limit with an integral, 
and (\ref{hpreln}),
we have
\begin{align}
\label{blimit1}
&\lim_{{\rm Im}{(b)}\to\pm\infty}
\Gamma\bigg[{1+a-h,b-a+(c,d,e,f,g)
\atop
b-a}\bigg]
\MM(a; b; c, d, e, f, g, h)\nn\\ 
&=\frac{1} {\Gamma[c,d,e,f,g,2+2a-c-d-e-f-g ]}\nn\\
&\times  \frac{1}{2\pi i}  
\int_t \Gamma\biggl[{a+t, 1+\frac{1}{2}a+t, t+(c,d,e,f,g), -t\atop \frac{1}{2}a+t, 1+a+t-(c,d,e,f,g)}\biggr]\nn\\
&\times\biggl[ \lim_{{\rm Im}{(b)}\to\pm\infty}\Gamma\biggl [ {t+b, t+h, b-a-t,1+a-h\atop b,h,1+a+t-h,b-a}\biggr] \,dt\nn\\  
&=\frac{1} {\Gamma[c,d,e,f,g,2+2a-c-d-e-f-g ]}\nn\\
&\times  \frac{1}{2\pi i}  
\int_t \Gamma\biggl[{a+t, 1+\frac{1}{2}a+t, t+(c,d,e,f,g), -t\atop \frac{1}{2}a+t, 1+a+t-(c,d,e,f,g)}\biggr] \nn\\
&\times\biggl( \lim_{{\rm Im}{(b)}\to\pm\infty}\Gamma\biggl[{t+h, b-a- t\atop h,b-a}\biggr] \biggr)\,dt
\end{align}
since, by Lemma \ref{L610} 
(and the fact that, by (\ref{hpreln}), $1+a-h$ equals a constant plus $b$),
$$\lim_{{\rm Im}{(b)}\to\pm\infty}\Gamma\biggl[{ t+b, 1+a-h \atop b ,1+a+t-h }\biggr]
=\lim_{{\rm Im}{(b)}\to\pm\infty}\frac{ (b)_t }{(1+a-h)_t}=1.$$
To evaluate the remaining limit in  (\ref{blimit1}) we note that, by (\ref{gammaref}),
\begin{align*} 
&\lim_{{\rm Im}{(b)}\to\pm\infty} \Gamma\biggl[{ t+h, b-a-t \atop h, b-a }\biggr] =  \lim_{{\rm Im}{(b)}\to\pm\infty}\frac{ \Gamma[ t+h, 1+a-b ] }{ \Gamma[ h,1+t+a-b]} 
\frac{\sin\pi(b-a)}{\sin\pi(b-a-t)}\\
&=  \lim_{{\rm Im}{(b)}\to\pm\infty}\frac{ (h)_t   }{ (1+a-b)_t} 
\frac{\sin\pi(b-a)}{\sin\pi(b-a-t)} =e^{\mp i \pi t}, 
\end{align*}
the last step by Lemma \ref{L610}, 
by (\ref{hpreln}), 
and by the fact that 
\begin{equation*}
\frac{\sin\pi x}{\sin\pi(x-y)}\to e^{\mp i \pi y}
\mbox{ as } 
\im{x}\to\pm\infty.
\end{equation*}

So by (\ref{blimit1}), we have 
\begin{align*}
&\lim_{{\rm Im}{(b)}\to\pm\infty}
\Gamma\bigg[{1+a-h,b-a+(c,d,e,f,g)
\atop
b-a}\bigg]
\MM(a; b; c, d, e, f, g, h)\\ 
&=\frac{1} {\Gamma[c,d,e,f,g,2+2a-c-d-e-f-g ]}\\
&\times  \frac{1}{2\pi i}  
\int_t \Gamma\biggl[{a+t, 1+\frac{1}{2}a+t, t+(c,d,e,f,g), -t\atop \frac{1}{2}a+t, 1+a+t-(c,d,e,f,g)}\biggr]e^{\mp i \pi t}\,dt.
\end{align*}
To evaluate the integral in $t$, we move the line of integration all the way to the right, 
and sum residues at the simple poles 
$t=n$ ($n\in\Z_{\ge0}$)  of the function $\gg{-t}$.  
In this way, we find that 
\begin{align*}
&\lim_{{\rm Im}{(b)}\to\pm\infty}
\Gamma\bigg[{1+a-h,b-a+(c,d,e,f,g)
\atop
b-a}\bigg]
\MM(a; b; c, d, e, f, g, h)\\ 
&=  \frac{\Gamma[a, 1+\frac{1}{2}a ]}{\Gamma[{\frac{1}{2}a, 1+a-(c,d,e,f,g),2+2a-c-d-e-f-g}]}\\
&\times  {}_7F_6\biggl[{a,1+\frac{1}{2}a,c,d,e,f,g  ;\atop \frac{1}{2} a, 1+a-c,1+a-d,1+a-e,1+a-f,1+a-g;}1\biggr]. \end{align*}
But then, by (\ref{eq240}), 
which relates the above 
very-well-poised ${}_7F_6(1)$ series to an $\LL$ function, we have
\begin{align}
\label{blim3}
&\lim_{{\rm Im}{(b)}\to\pm\infty}
\Gamma\bigg[{1+a-h,b-a+(c,d,e,f,g)
\atop
b-a}\bigg]
\MM(a; b; c, d, e, f, g, h)\nn\\ 
&=    \pi\Gamma\biggl[{ a, 1+\frac{1}{2}a  \atop\frac{1}{2}a, 1+a} \biggr]
\LL\biggl[{ 1+a-d-e,1+a-d-f,1+a-d-g,c;\atop1+c-d;2+2a-d-e-f-g,1+a-d }\biggr]. 
\end{align}

We apply the functional equation 
$\Gamma(1+s)=s\Gamma(s)$
to the gamma factors on 
the right-hand side of (\ref{blim3}), 
and apply \cite[Eq.\ (4.10)]{M1} to the $\LL$ function there, to get
\begin{align*}
&\lim_{{\rm Im}{(b)}\to\pm\infty}
\Gamma\bigg[{1+a-h,b-a+(c,d,e,f,g)
\atop
b-a}\bigg]
\MM(a; b; c, d, e, f, g, h)\\ 
&=    \frac{\pi}{2} 
\LL\biggl[{ e,f,g,1+a-c-d;\atop e+f+g-a; 1+a-c,1+a-d }\biggr],  
\end{align*}
and \eqref{MlimitL} is proved.

We next prove \eqref{MlimitJ}.  
By the Barnes integral representation (\ref{MBarInt}) of the $\MM$ function, 
and by an interchange of a limit and an integral (which may readily be justified here), 
we have
\begin{align*}
&\lim_{{\rm Im}{(g)}\to \pm\infty} \Gamma[1+a-(g,h),b-a+(g,h)]
\MM(a; b; c, d, e, f, g, h)\\
&=  \frac{1} {\Gamma[b,c,d,e,f,b-a+(c,d,e,f)]}\\
&\times\frac{1}{2\pi i}  
\int_t \Gamma\biggl[{a+t, 1+\frac{1}{2}a+t, t+(b,c,d,e,f), b-a-t, -t\atop \frac{1}{2}a+t, 1+a+t-(c,d,e,f)}\biggr] \\
&\times\biggl(\lim_{{\rm Im}{(g)}\to
\pm\infty
}\Gamma\biggl[{ 1+a-g,1+a-h,t+g,t+h\atop g,h,1+a+t-g,1+a+t-h}\biggr]\biggr)\,dt\\
&= \frac{1} {\Gamma[b,c,d,e,f,b-a+(c,d,e,f)]}\\
&\times\frac{1}{2\pi i}  
\int_t \Gamma\biggl[{a+t, 1+\frac{1}{2}a+t, t+(b,c,d,e,f), b-a-t, -t]\atop\frac{1}{2}a+t, 1+a+t-(c,d,e,f)}\biggr] \\
&\times\biggl[\lim_{{\rm Im}{(g)}\to
\pm\infty}\frac{ (g)_t(h)_t}{(1+a-g)_{t}(1+a-h)_{t}}\biggr]\,dt.
\end{align*}
The limit in square brackets equals $1$, by the above Lemma \ref{L610} 
(since, by (\ref{hpreln}), $1+a-g=h+x$ and $1+a-h=g+x$, 
for a complex number $x$ independent of $g$ and $h$).  So
\begin{align}
\label{nextlim}
&\lim_{{\rm Im}{(g)}\to \pm\infty} 
\Gamma[1+a-(g,h),b-a+(g,h)]
\MM(a; b; c, d, e, f, g, h)\nn\\
&= \frac{1} {\Gamma[b,c,d,e,f,b-a+(c,d,e,f)]}\nn\\
&\times\frac{1}{2\pi i}  
\int_t \Gamma\biggl[{a+t, 1+\frac{1}{2}a+t, t+(b,c,d,e,f), b-a-t, -t\atop\frac{1}{2}a+t, 1+a+t-(c,d,e,f)}\biggr]\, dt.
\end{align}
But, by Barnes’ Second Lemma (see \cite[Eq.\ 6.2.1]{Ba}),
\begin{align*}
&\Gamma\biggl[{ t+(d,e,f)\atop 1+a+t-(d,e,f)}\biggr] 
=\frac{1}{\Gamma[1+a-((d,e,f))]}\\
&\times\frac{1}{2\pi i}\int_s 
\Gamma\biggl[{s+(d,e,f),1+a-d-e-f-s,t-s\atop 1+a+t+s}\biggr]\,ds,
\end{align*}
so by (\ref{nextlim}),
\begin{align}
\label{doubint}
&\lim_{{\rm Im}{(g)}\to \pm\infty}
\Gamma[1+a-(g,h),b-a+(g,h)]
\MM(a; b; c, d, e, f, g, h)\nn\\
&=\frac{1}{\Gamma[b,c,d,e,f,b-a+(c,d,e,f),1+a-((d,e,f))]} \nn\\
&\times\frac{1}{2\pi i}  
\int_t \Gamma\biggl[{a+t, 1+\frac{1}{2}a+t, t+(b,c),  b-a-t, -t\atop  \frac{1}{2}a+t, 1+a+t-c}\biggr] \nn\\&
\times\frac{1}{2\pi i}\int_s \Gamma\biggl[{s+(d,e,f),1+a-d-e-f-s,t-s\atop 1+a+t+s }\biggr]\,ds\,dt\nn\\
&=\frac{1}{\Gamma[b,c,d,e,f,b-a+(c,d,e,f),1+a-((d,e,f))]}\nn\\
&\times \frac{1}{2\pi i}  
\int_s  {\Gamma[s+(d,e,f),1+a-d-e-f-s]} \nn\\
&\times\frac{1}{2\pi i}\int_t \Gamma\biggl[{ a+t, 1+\frac{1}{2}a+t, -s+t, t+(b,c),  b-a-t, -t\atop \frac{1}{2}a+t,1+a+s+t ,1+a-c+t}\biggr]\,dt\,ds,
\end{align}
the switch in the order of integration being readily justified by Fubini’s Theorem.  

By \cite[Eq.\ 6.6.1]{Ba}, the above integral in $t$ 
(including the factor of $(2\pi i)^{-1}$ in front) equals
$$ \frac{1}{2}\Gamma\biggl[{b,c,-s, b+c-a,b-a-s\atop1+a-c+s,b+c-a-s}\biggr],$$
so (\ref{doubint})  gives
\begin{align}
\label{K1}
&\lim_{{\rm Im}{(g)}\to \pm\infty} 
\Gamma[1+a-(g,h),b-a+(g,h)]
\MM(a; b; c, d, e, f, g, h)\nn\\
&=\frac{1}{2\,\Gamma[d,e,f,b-a+(d,e,f),1+a-((d,e,f))]}\nn\\
&\times \frac{1}{2\pi i}
\int_s \Gamma\biggl[{s+(d,e,f),b-a-s,1+a-d-e-f-s,-s\atop 1+a-c+s,b+c-a-s}\biggr]\,ds\nn\\
&=\frac{1}{2\,\Gamma[d,e,f,b-a+(d,e,f),1+a-((d,e,f))]}\nn\\
&\times \frac{1}{2\pi i}
\int_s \frac{\left[\displaystyle{
\Gamma[b+c-a+s+(d,e,f)]
\atop
\times \, \Gamma[-c-s,1+2a-b-c-d-e-f-s,a-b-c-s]}\right]
}
{\Gamma[1+b+s,-s]}\,ds,
\end{align}
the last step by the substitution $s\to s+b+c-a$.   
By the formula for the $K$ function given in the proof of 
\cite[Proposition 7.3]{FGS}, the right-hand side of (\ref{K1}) equals
\begin{align}
\label{K2} 
\frac{1}{2} 
 K\biggl[{b;c,b+c-a,b+c+d+e+f-2a-1;\atop b+c+d-a,b+c+e-a,b+c+f-a}\biggr];
\end{align}
by  \cite[Corollary 3.2]{FGS},  (\ref{K2})  is in turn equal to
\begin{align}
\label{K3} 
\frac{1}{2} 
  K\biggl[{b;c,d,1+a-e-f;\atop b+c+d-a,1+a-e,1+a-f}\biggr].
\end{align}
Finally, by (\ref{JandK}) and (\ref{gammaref}), 
we have that (\ref{K3}) is equal to
\begin{equation*}
\frac{\pi}{2 \, \Gamma(1-b)}   
\JJ\biggl[{b;c,d,1+a-e-f;\atop b+c+d-a,1+a-e,1+a-f}\biggr],
\end{equation*}
and our proof is complete.

\end{proof}

The above proposition provides, on the surface, two different limits of the function $M(\vw)$ (suitably normalized). But  a closer analysis reveals the following. We may, in fact, evaluate the limit as $\hbox{Im}(b)\to\infty$ of {\it any one} of the 56 functions $M_\tau(\vw)\  (\tau \in G \backslash H)$ (suitably normalized), if we apply an appropriate change of variable to the proposition. Moreover, such a limit is an $L $  or $J$ function, depending on the orbit to which $\tau$ belongs.  Finally, the $J/L$ function thus obtained is exactly the one specified by the correspondence $\gamma_1$ (cf. Definition \ref{CosetLabels}).

To see that all of the above is true, let us write
\begin{equation}\alpha\vw=(a_\alpha,b_\alpha,c_\alpha,d_\alpha,e_\alpha,f_\alpha,g_\alpha,h_\alpha)\in W,\label{subalphas}\end{equation}for $\alpha\in H\hbox{ and }\vw=(a ,b ,c ,d ,e ,f ,g ,h )\in W$
(see (\ref{Wdef}) for the definition of the affine hyperplane $W$). 
We have the following.

 \begin{Lemma}\label{good_w}
\begena\item[\rm(a)] Let $\tau\in \mathscr{O}_1$.  Then there is some $\alpha\in\tau$ such that $\hbox{Im}(b_\alpha)$  tends to $+\infty$  as $\hbox{Im}(b)$ does, and such that $a_\alpha,c_\alpha,d_\alpha,e_\alpha,f_\alpha$, and $g_\alpha$ are independent of $b$ (and of $h$).   
\item[\rm(b)] Let $\tau\in \mathscr{O}_2$.  Then there is some $\alpha\in\tau$ such that $\hbox{Im}(b_\alpha)$  tends to $-\infty$  as $\hbox{Im}(b)\to+\infty$, and such that $a_\alpha,c_\alpha,d_\alpha,e_\alpha,f_\alpha$, and $g_\alpha$ are independent of $b$ (and $h$).   

\item[\rm(c)] Let $\tau\in \mathscr{O}_3$.  Then there is some $\alpha\in\tau$ such that $\hbox{Im}(g_\alpha)$  tends to $+\infty$  as $\hbox{Im}(b)$ does, and such that $a_\alpha,b_\alpha,c_\alpha,d_\alpha,e_\alpha $, and $f_\alpha$ are independent of $b$ (and $h$).    

\item[\rm(d)] Let $\tau\in \mathscr{O}_1\cup \mathscr{O}_2$, respectively $\tau\in \mathscr{O}_3$.  Then the limit, as $\hbox{Im}(b)\to+\infty,$ of $M_\tau(\vw)$ times an appropriate ratio of gamma functions equals an $L$ function, respectively a $J$ function, of $a,c,d,e,f,g$. \enden 
\end{Lemma}

\begin{proof} Parts (a), (b), and (c) may be verified explicitly — see the Appendix below, where each $M_\tau(\vw)$ is written explicitly in the form$$M_{\tau}(\vw)=M(\alpha\vw)=(a_\alpha,b_\alpha,c_\alpha,d_\alpha,e_\alpha,f_\alpha,g_\alpha,h_\alpha),$$with $a_\alpha,b_\alpha,\ldots, \alpha_h$ having the stated properties.  Note that verification requires direct calculation in only a few cases; the remaining cases follow from simple changes of variable.  For example:  part (b) follows from part (a) and the interchange of $b$ with $h$.  Within orbit $\mathscr{O}_1$, the last six cases, as delineated in the Appendix, follow from the first six and the change of variable $\vw\leftrightarrow Z(b,h)\vw$, where $Z$ is the central involution \eqref{Zaction}.  Within those first six, five are obtained from each other by permutations of $c,d,e,f,g$. And so on.  Similar bookkeeping applies to the analysis of $\mathscr{O}_3$.

Part (d) follows from parts (a), (b), and (c).
\end{proof}

We make the following.

\begin{Definition}\label{gamma2def}  Let $\tau\in G\backslash H$.

We denote by $\gamma_2(M_\tau)(a,c,d,e,f,g)$ the $J/L$ function obtained from $M_\tau(\vw)$, by way of Proposition \ref{limitsareLorJ} and Lemma \ref{good_w} above. \end{Definition}

The next theorem tells us that the association $M_\tau\rightarrow \gamma_2(M_\tau)$ is a familiar one.

\begin{Theorem}\label{limit_theorem}

For $\vw=(a,b,c,d,e,f,g,h)\in W$, define $\vx\in V $ by
\begin{align}&\vx=\vx(\vw)=\vx(a,c,d,e,f,g)=(A,B,C,D,E,F,G)^T \nr&=\biggl({2+2a-c-d-e-f-g,1+a-e-f,  1+a-e-g,  1+a-f-g, \atop 2+2a-d-e-f-g,  2+2a-c-e-f-g,  2+a-e-f-g}\biggr)^T.\label{xfromw}\end{align}Then, for any $\tau\in G\backslash H$, we have \begin{equation}\gamma_2(M_\tau)(a,c,d,e,f,g) =\begin{cases} L_{\gamma_1(\tau)}(\vx)&\text{if $\tau\in\mathscr{O}_1\cup\mathscr{O}_2$},\\ J_{\gamma_1(\tau)}(\vx)&\text{if $\tau\in\mathscr{O}_3$.}\end{cases}\label{corresp}\end{equation}
\end{Theorem}

\begin{proof} In light of Lemma \ref{good_w}, it suffices to verify the following two statements.

\begen\item[(i)] The equality \eqref{corresp} holds for one element of each orbit — say, for $\tau_1=v(0,7)\in \mathscr{O}_1$, $\tau_2=-v(0,7)\in \mathscr{O}_2$, and $\tau_3=v(0,1)\in \mathscr{O}_3$.  

\item[(ii)] Translating $M_{\tau_i}$ by an element $q\in Q$, and then mapping the result to a $J$ or $L$ function via $\gamma_1$,   yields the same result as applying $m(q)\in H_1$ to the image $L_{\gamma_1(\tau_i)}(\vx)$ or $J_{\gamma_1(\tau_i)}(\vx)$ (with $\vx$ as in \eqref{xfromw}) of $M_{\tau_i}$ under $\gamma_1$.  Here $Q$ and $m$ are as in Proposition \ref{3orbits}.

\enden 
 
To prove (i), we first note that the $\mathscr{O}_2$ case follows from the $\mathscr{O}_1$ case, and the change of variable $(b,h)\leftrightarrow (h,b)$.  So we need only consider the two cases $\tau_1=v(0,7)\in \mathscr{O}_1$  and $\tau_3=v(0,1)\in \mathscr{O}_3$.
 
For the first case, let $\beta\in\tau_1$. The second entry of  $\beta \vw$ equals $b$, by Definition \ref{vij}(b).  We therefore have$$M_{\tau_1}(\vw)=M\biggl[{a;b;c,d,\atop e,f,g,h}\biggr].$$From this and from \eqref{MlimitL}, we find that \begin{equation}\label{v13}\gamma_2({M_{\tau_1}})(a,c,d,e,f,g)= \LL\biggl[{e,f,g,1+a-c-d;\atop e+f+g-a;1+a-c,1+a-d}\biggr].\end{equation}
On the other hand, we have $$\gamma_1(\tau_1)=\gamma_1(v(0,7))=6,$$by Definition \ref{CosetLabels}(a).  By \eqref{Lcosetsdef}, then, we have \begin{align}\label{l2x}L_{\gamma_1(\tau_1)}(\vx)&=L_6(\vx)=L\biggl[{A,B,C,D;  \atop G;F,E}\biggr]\nr&=
L\biggl[{2+2a-c-d-e-f-g,1-c,1-d,1+a-c-d; \atop 2+a-c-d-f;2+a-c-d-g,2+a-c-d-e}\biggr],\end{align}the last step by \eqref{xfromw}.   The equality of \eqref{v13} and \eqref{l2x} follows from the $L$-function invariance given by \cite[Eq. (4.8)]{M1}, and we have verified statement (i) in the case at hand.  (See also formula \eqref{v07} below.) 

We now consider the case $\tau_3= v(0,1)\in\mathscr{O}_3$.  On the one hand we compute $\gamma_2(M_{\tau_3})$, as follows.  The second entry of  $\beta \vw$, for  $\beta\in\tau$, equals $$b+h-a=2+2a-c-d-e-f-g,$$ by Definition \ref{vij}(c).  In fact, we may show that$$M_{\tau_3}(\vw)=M\biggl[{  2+a-c-d-e-f;2+2a-c-d-e-f-g;1-c, 1-d, 
\atop
1-e,1-f,b+g-a,h+g-a}\biggr].$$From this, and from a change of variable applied to \eqref{MlimitJ}, we find that \begin{align}\label{minusv46a}&\gamma_2({M_{\tau_3}})(a,c,d,e,f,g)\nr&= \JJ\biggl[{2+2a-c-d-e-f-g;1-c,1-d,1+a-c-d;\atop 2+a-c-d-e,2+a-c-d-f,2+a-c-d-g}\biggr].\end{align}
On the other hand, we have $$\gamma_1(\tau_3)=\gamma_1(v(0,1))={+}{+}{+}{+}{+}{+}=p_{0},$$by Definition \ref{CosetLabels}(c).  By Definition \ref{J_coset_labels}(b), then, the first coordinate of $\alpha \vx$, for $\alpha\in \gamma_1(\tau_1)$, is $A$, which equals $2+2a-c-d-e-f-g$ by \eqref{xfromw}.   That is, the first argument of $\JJ_{\gamma_1(\tau_1)}(\vx)$ is $2+2a-c-d-e-f-g$, which is the first coordinate of the $J$ function in \eqref{minusv46a}.  But recall, cf. the discussion preceding Definition \ref{J_coset_labels}, that the first argument of $J_\sigma(\vx)$ uniquely determines  $\sigma$. Statement (i) is thereby verified in this case as well.   (See also formula \eqref{v01} below.)

To prove statement (ii) above, it suffices to restrict our attention to the generators $s_1,s_2,\ldots,s_5$ and $s_{3’}$ of $Q$.  Consider, for example, the case $q=s_{3’}$ and $i=3$.  On the one hand, we have  $M_{s_{3’}\tau_3}=M_{-v(2,3)}$, by Proposition \ref{M_action_labels}(b).  The $J$ function associated to the latter $M$ function, through the correspondence $\gamma_1$, is \begin{equation} \label{newJ} J_{\gamma_1(-v(2,3))}(\vx)=J_{{-}{-}{+}{+}{+}{+}}(\vx)=J_{p_3}(\vx),\end{equation}whose first argument is $D$, which equals $1+a-f-g$ by \eqref{xfromw}.

On the other hand, as we have shown above, the $J$ function associated to $M_{\tau_3}$ is $$J_{\gamma_1(\tau_3)}(\vx)=J_{p_0}(\vx)=J(A;B,C,D;E,F,G).$$ Applying, to the latter, the transformation $$m(s_{3’})=a_{1’}=(14)$$yields $J(D;B,C,A;E,F,G)$, which again has first argument equal to $D=1+a-f-g$, and is therefore the same $J$ function as \eqref{newJ}.  So we have verified the above statement (ii) in the case $q=s_{3’}$ and $i=3$.
 
Other cases of (ii) may demonstrated in similar ways, exploiting symmetries to reduce the amount of computation.  (For example, each case in orbit $\mathscr{O}_2$ follows from one in orbit $\mathscr{O}_1$ and the interchange of $b$ and $h$.)\end{proof}

The correspondence $\gamma_1$ (equivalently, $\gamma_2$) is summarized explicitly in the Appendix below.

\section{$J/L$ function relations as limits of $M$ function relations} 

 Further   connections between the set $\{ M_\tau\colon \tau\in G\backslash H\}$  on the one hand (cf. \S 3), and  the sets $\{ \JJ_\sigma\colon \sigma\in G_\JJ\backslash H_1\}$ and $\{ \LL_\mu\colon \mu\in G_\LL\backslash H_1\}$ on the other (cf. \S 4), may be seen by considering three-term relations among the corresponding $M$, $J$, and $L$ functions. 
 
Before stating the main result along these lines, we consider a particular example, which illustrates the central ideas.

\begin{Proposition}\label{limit222} If the Euclidean type 222 $M$-function relation 
\begin{align}
&\frac{\sin\pi(b-a)\MM_{v(0,7)}(\vw) }
{\Gamma[c-a+(d,e,f,g,h)]}
+\frac{\sin\pi(a-c)\MM_{v(6,7)}(\vw)}
{\Gamma[b-a+(d,e,f,g,h)]}+\frac{\sin\pi(c-b)\MM_{v(0,6)}(\vw)}
{\Gamma[d,e,f,g,h]}
=0,
\label{Roy463}\end{align}which appears as \cite[Eq. (4.6.3)]{Roy}, is multiplied by an appropriate ratio of Gamma functions, and the limit as $\hbox{Im}(b)\to\infty$ of the result is taken, then one obtains the  $(J,L,L)$ relation
of  \cite[Eq. (6.2)]{GMS}:\begin{align}\label{Orbit1JLL}
 &\frac{\sin\pi(F-E) \ser{J}{p_0} }{\Gamma[1 - A,E-A,F-A,G-A]} + \frac{\sin\pi(F-A)\ser{L}{4}}{\Gamma[E-A,E-B,E-C,E-D]} \nr-\,&  \frac{\sin\pi(E-A)\ser{L}{5}}{\Gamma[F-A,F-B,F-C,F-D] }
=0,
\end{align}where \begin{align}\vx=(A,B,C,D,E,F,G)=\biggl({c,1+a-d-g,1+a-e-g,1+a-f-g,\atop 1+c-g,1+a-g,2+2a-d-e-f-g}\biggr)^T.\label{newxdef}\end{align}

\end{Proposition}

\begin{proof}Referring to \eqref{v07}, \eqref{v06}, and \eqref{v67} below, we see that \eqref{Roy463} takes the form
\begin{align*}
&\frac{\sin\pi(b-a) }
{\Gamma[c-a+(d,e,f,g,h)]}\MM \biggl[ {a;b;c,d,\atop e,f,g,h}\biggr]+\frac{\sin\pi(a-c) }
{\Gamma[b-a+(d,e,f,g,h)]}M\biggl[{a;c;b,d,\atop e,f,g,h}\biggr]\nn\\ 
+&\frac{\sin\pi(c-b) }
{\Gamma[d,e,f,g,h]} M\biggl[{2c-a;c+b-a;c,c+d-a,\atop c+e-a,c+f-a,c+g-a,c+h-a}\biggr]
=0.
 \end{align*}Applying the invariance of the second $M$ function under $b\leftrightarrow g$,   as well as the fact that $b+h-a =2+2a-c-d-e-f-g$, and multiplying through by $$\frac{2}{\pi}\Gamma\biggl[{1+a-(b, h),b-a+(c,d,e,f,g),c-a+h\atop 1-g,1+a-c-g}\biggr],$$ we get (after some rearranging, and making use of the reflection formula \eqref{gammaref}):
\begin{align}&\frac{2\sin\pi(c+g-a)}{\pi\Gamma[1-g,c-a+(d,e,f)]}\nr\times
&\biggl(\Gamma\bigg[{1+a-h,b-a+(c,d,e,f,g)
\atop
b-a}\bigg]\MM \biggl[ {a;b;c,d,\atop e,f,g,h}\biggr]\biggr)\nr
+& \frac{2\sin\pi(a-c) }{\pi \Gamma[{1-c,
2+2a-c-d-e-f-g,1-g,1+a-c-g}]} \nr\times&\biggl( \Gamma[1-c,1+a-(b,h),c-a+(b,h) ] M\biggl[{a;c;g,d,\atop e,f,b,h}\biggr]\biggr)\nr 
-&  \frac{2\sin\pi g}{\pi \Gamma[1+a-c-g,
  d,e,f]}\biggl\{\frac{(b)_{c-a}(h)_{c-a}}{(1+a-b)_{c-a}(1+a-h)_{c-a}}\biggr\} \nr\times&\biggl(\Gamma\biggl[ {1+c-h,b,b-a+(d,e,f,g)\atop b-c}\biggr]\nr\times &M\biggl[{2c-a;c+b-a;c,c+d-a,\atop c+e-a,c+f-a,c+g-a,c+h-a}\biggr]\biggr)
=0,\label{Roy463b}
 \end{align}where the symbol $(g)_a$ is as defined in \eqref{gsuba}.
 
 We now take limits as $\hbox{Im}(b)\to\infty.$ The quantity in curly braces, in \eqref{Roy463b}, becomes $1$, by Lemma \ref{L610}.  Moreover, the limit of each of the three terms in large parentheses, in \eqref{Roy463b}, may be evaluated by Proposition \ref{limitsareLorJ} and an appropriate change of variable.  The end result is that, in the limit as $\hbox{Im}(b)\to\infty$, \eqref{Roy463b} becomes
  \begin{align}
&\frac{\sin\pi(c+g-a)}{\Gamma[1-g,c-a+(d,e,f) ]}     \LL\biggl[{e,f,g,1+a-c-d;\atop e+f+g-a; 1+a-c,1+a-d}\biggr]\nr
+& \frac{\sin\pi(a-c) }{\Gamma[1-g,1+a-c-g, 1-c,
2+2a-c-d-e-f-g]} \nr\times& \JJ\biggl[{c;g,d,1+a-e-f;\atop c+d+g-a,1+a-e,1+a-f}\biggr] \nr 
-& \frac{\sin\pi g}{\Gamma[ 1+a-c-g, d,e,f]}     \LL\biggl[{ c+e-a,c+f-a,c+g-a,1-d;\atop c+e+f+g-2a; 1+c-a,1+c-d}\biggr]
=0.\label{Roy463blim2}
 \end{align}Making the change of variable \eqref{newxdef}, we see that \eqref{Roy463blim2} becomes
   \begin{align}
&\frac{\sin\pi(F-A)}{\Gamma[ E-A,E-B,E-C,E-D]} \nr\times&   \LL\biggl[{F-C,F-D,1+A-E,1+B-E;\atop 1+F-C-D; 1+F-E,1+A+B-E}\biggr]\nr
+& \frac{\sin\pi(F-E) }{\Gamma[ E-A,F-A,1-A,
G-A]}  \JJ\biggl[{A;1+A-E,F-B,G-B;\atop 1+A-B,1+A+C-E,1+A+D-E}\biggr] \nr 
-& \frac{\sin\pi(E-A)}{\Gamma[ F-A, F-B,F-C,F-D]}  
\nr\times&  \LL\biggl[{ E-C,E-D,1+A-F,1+B-F;\atop 1+E-C-D; 1+E-F,1+A+B-F}\biggr]
=0.\label{Roy463blim4}
 \end{align}But the $J/L$ functions in \eqref{Roy463blim4} are readily seen, by invariances of these functions, to equal $L_4(\vx)$, $J_{p_0}(\vx)$, and  $L_5(\vx)$ respectively, so \eqref{Roy463blim4} is exactly \eqref{Orbit1JLL}.
 \end{proof}
 
 \begin{Remark} \begena\item 
 In \cite[Eq. (6.2)]{GMS},  \eqref{Orbit1JLL} is in fact expressed as a relation among an $L_6$, a $J_{p_0}$, and  an $L_5$ series. But recall (cf. Remark \ref{LcosetsRem} above) that the roles of the $L$-coset indices $4$ and $6$ are reversed here, compared to  \cite{GMS}.  So the result expressed there is, in fact, the same as \eqref{Orbit1JLL}.
 
\item   The change of variable \eqref{newxdef} is not the one invoked in Theorem \ref{limit_theorem}.  Had we applied the latter to \eqref{Roy463blim2}, in the above proof, we would have obtained a $(J,L,L)$ relation different from \eqref{Orbit1JLL}, but in the same orbit (cf. Theorem \ref{bigthm}(a)(ii)).  Specifically, we would have obtained a relation among $L_4$, $L_5$, and $J_{n_4}$.    That such a relation belongs to the same orbit as \eqref{Orbit1JLL} may be seen from the proof of \cite[Proposition 4.3]{GMS}.
\enden \end{Remark}
The result of the above proposition is reflected in part (c)(i) of the following.
\begin{Theorem}
\label{mainthm}

Each of the eighteen types of  $\LL/\JJ$-function relations  described in Theorem \ref{bigthm}(b) may be obtained as a limit of one of the three-term $M$-function relations given in \cite{Roy}, together with an appropriate change of variable.

More specifically, we have the following.
\begin{enumerate}
\item[\rm(a)] The Hamming type $222$, $224$, $244$, $444$, and $246$ 
$J$-function relations come from Euclidean type
$222$, $224$, $244$, $444$, and $246$ 
$\MM$-function relations, respectively.

\item[\rm(b)] The $L$-coherent $L$-function relations
come from Euclidean type $222$ or $244$ $\MM$-function relations.
The $L$-incoherent $L$-function relations come from Euclidean type
$224, 444$, or $246$ $\MM$-function relations.

\item[\rm(c)] Under an appropriate numbering (cf. \cite[Proposition 4.3]{GMS}) of the four orbits described in Theorem \ref{bigthm}(a)(ii), we have the following.  \begin{enumerate}
\item[\rm(i)] The Orbit 1 $(J,L,L)$ relations come from Euclidean type
$222$ or $224$ $\MM$-function relations.
\item[\rm(ii)] The Orbit 2 $(J,L,L)$ relations come from Euclidean type
$224$ or $244$ $\MM$-function relations.
\item[\rm(iii)] The Orbit 3 $(J,L,L)$ relations come from Euclidean type
$244$ or $246$ $\MM$-function relations.
\item[\rm(iv)] The Orbit 4 $(J,L,L)$ relations come from Euclidean type
$244$ or $444$ $\MM$-function relations.
\end{enumerate}

\item[\rm(d)] Under an appropriate numbering (cf. \cite[Proposition 4.4]{GMS}) of the seven orbits described in Theorem \ref{bigthm}(a)(iii), we have the following. 
\begin{enumerate}
\item[\rm(i)] The Orbit 1 $(L,J,J)$ relations come from Euclidean type
$222$ $\MM$-function relations.
\item[\rm(ii)] The Orbit 2 $(L,J,J)$ relations come from Euclidean type
$244$ $\MM$-function relations.
\item[\rm(iii)] The Orbit 3 $(L,J,J)$ relations come from Euclidean type
$224$ $\MM$-function relations.
\item[\rm(iv)] The Orbit 4 $(L,J,J)$ relations come from Euclidean type
$224$ $\MM$-function relations.
\item[\rm(v)] The Orbit 5 $(L,J,J)$ relations come from Euclidean type
$444$ $\MM$-function relations.
\item[\rm(vi)] The Orbit 6 $(L,J,J)$ relations come from Euclidean type
$244$ $\MM$-function relations.
\item[\rm(vii)] The Orbit 7 $(L,J,J)$ relations come from Euclidean type
$246$ $\MM$-function relations.
\end{enumerate}

\enden
\end{Theorem}

\begin{proof} The theorem can be demonstrated by arguments similar to those of Proposition \ref{limit222} above. 

The amount of computation required can be reduced dramatically by observing that existing three-term relations, either among $M$ functions or among $J/L$ functions, can be used to obtain new ones.  Indeed, in \cite{Roy}, the relation \eqref{Roy463} is combined with translates of (that is, changes of variable applied to) itself, to ultimately yield all three-term relations.  Similar techniques are used to build new $J/L$ relations out of old ones, in \cite{FGS}, \cite{M1}, and \cite{GMS}.  We can thereby ``bootstrap'' from the result of Proposition \ref{limit222} to other cases of Theorem \ref{mainthm}.  We omit the details. \end{proof}

\begin{Remark}  A three-term $\MM$-function relation involving a blue $L$ coset $\tau_1$ and a red $L$ coset $\tau_2$ of the same  label  yields a trivial, or ``$0=0$,'' relation in the limit as ${\rm Im}(b)\to\infty$.  More specifically:  if  $\tau_1$ and $\tau_2$ are such cosets and, for some coset $\tau_3\in G\backslash H$, we have a suitably normalized relation
\begin{equation} \label{relngeneralM}
d_1(\vw) M_{\tau_1} (\vw ) + d_2(\vw) M_{\tau_2} (\vw) + d_3(\vec w) M_{\tau_3} (\vw) = 0,
\end{equation} then one can show that$$\lim_{{\rm Im}(b)\to\infty}d_1(\vw) =-\lim_{{\rm Im}(b)\to\infty}d_2(\vw) \hbox{\rm\quad and\quad }\lim_{{\rm Im}(b)\to\infty}d_3(\vw) =0.$$

 \end{Remark}

\setcounter{section}{1}

\section*{Appendix. Explicit form of elements of the orbits $\mathscr{O}_1$, $\mathscr{O}_2$, and $\mathscr{O}_3$, with corresponding $\JJ/\LL$ functions}
\label{threeorbitslist}

\setcounter{equation}{0}
\renewcommand{\theequation}{A.\arabic{equation}}

\subsection*{Orbit \boldmath $\mathscr{O}_1$ (blue $L$)} 
\begin{align}  M_{v(0,7)}(\vw)=M &\biggl[{a;b;c,d,\atop e,f,g,h}\biggr] \nr \leftrightarrow L_6(\vx)=L&\biggl[{e,f,g,1+a-c-d;\atop e+f+g-a;1+a-c,1+a-d}\biggr];\label{v07}
\\ M_{v(0,6)}(\vw)=M&\biggl[{2c-a;c+b-a;c,c+d-a,\atop c+e-a,c+f-a,c+g-a,c+h-a}\biggr] \nr\leftrightarrow L_5(\vx)=L&\biggl[{ c+e-a,c+f-a,c+g-a,1-d; \atop c+e+f+g-2a;1+c-a,1+c-d}\biggr];\label{v06}\\
M_{v(0,5)}(\vw)=M&\biggl[{2d-a;d+b-a;d,d+c-a,\atop d+e-a,d+f-a,d+g-a,d+h-a}\biggr] \nr\leftrightarrow L_4(\vx)=L&\biggl[{d+e-a,d+f-a,d+g-a,1-c;\atop d+e+f+g-2a;1+d-a,1+d-c}\biggr]\nrs 
M_{v(0,4)}(\vw)=M&\biggl[{2e-a;e+b-a;e,e+c-a,\atop e+d-a,e+f-a,e+g-a,e+h-a}\biggr] \nr\leftrightarrow L_3(\vx)=L&\biggl[{e+d-a,e+f-a,e+g-a,1-c;\atop e+d+f+g-2a;1+e-a,1+e-c}\biggr]
\nrs
M_{v(0,3)}(\vw)=M&\biggl[{2f-a;f+b-a;f,f+c-a,\atop f+d-a,f+e-a,f+g-a,f+h-a}\biggr]  \nr\leftrightarrow L_2(\vx)=L&\biggl[{f+d-a,f+e-a,f+g-a,1-c;\atop f+d+e+g-2a;1+f-a,1+f-c}\biggr]
\nrs
M_{v(0,2)}(\vw)=M&\biggl[{2g-a;g+b-a;g,g+c-a,\atop g+d-a,g+e-a,g+f-a,g+h-a}\biggr]  \nr\leftrightarrow L_1(\vx)=L&\biggl[{g+d-a,g+e-a,g+f-a,1-c;\atop g+d+e+f-2a;1+g-a,1+g-c}\biggr]
\nrs
M_{-v(1,7)}(\vw)=M&\biggl[{1-a;1-h;1-c,1-d,\atop 1-e,1-f,1-g,1-b}\biggr] \nr\leftrightarrow L_{\overline6}(\vx)=L&\biggl[{1-e,1-f,1-g,c+d-a;\atop 2+a-e-f-g;1+c-a,1+d-a}\biggr]
\nrs
M_{-v(1,6)}(\vw)=M&\biggl[{1+a-2c;1+a-c-h;1-c,1+a-c-d,\atop 1+a-c-e,1+a-c-f,1+a-c-g,1+a-c-b}\biggr]  \nr\leftrightarrow L_{\overline5}(\vx)=L&\biggl[{1+a-c-e,1+a-c-f,1+a-c-g,d;\atop 2+2a-c-e-f-g;1+a-c,1+d-c}\biggr]
;
\end{align}
\begin{align}
M_{-v(1,5)}(\vw)=M&\biggl[{1+a-2d;1+a-d-h;1-d,1+a-d-c,\atop 1+a-d-e,1+a-d-f,1+a-d-g,1+a-d-b}\biggr]  \nr\leftrightarrow L_{\overline4}(\vx)=L&\biggl[{1+a-d-e,1+a-d-f,1+a-d-g,c;\atop 2+2a-d-e-f-g;1+a-d,1+c-d}\biggr]
\nrs M_{-v(1,4)}(\vw)=M&\biggl[{1+a-2e;1+a-e-h;1-e,1+a-e-c,\atop 1+a-e-d,1+a-e-f,1+a-e-g,1+a-e-b}\biggr]  \nr\leftrightarrow L_{\overline3}(\vx)=L&\biggl[{1+a-e-d,1+a-e-f,1+a-e-g,c;\atop 2+2a-e-d-f-g;1+a-e,1+c-e}\biggr]
\nrs
M_{-v(1,3)}(\vw)=M&\biggl[{1+a-2f;1+a-f-h;1-f,1+a-f-c,\atop 1+a-f-d,1+a-f-e,1+a-f-g,1+a-f-b}\biggr]  \nr\leftrightarrow L_{\overline2}(\vx)=L&\biggl[{1+a-f-d,1+a-f-e,1+a-f-g,c;\atop 2+2a-f-d-e-g;1+a-f,1+c-f}\biggr]
\nrs
M_{-v(1,2)}(\vw)=M&\biggl[{1+a-2g;1+a-g-h;1-g,1+a-g-c,\atop 1+a-g-d,1+a-g-e,1+a-g-f,1+a-g-b}\biggr]  \nr\leftrightarrow L_{\overline1}(\vx)=L&\biggl[{1+a-g-d,1+a-g-e,1+a-g-f,c;\atop 2+2a-g-d-e-f;1+a-g,1+c-g}\biggr].
\end{align}
 
\subsection*{Orbit \boldmath$\mathscr{O}_2$ (red $L$)}
  
\begin{align}
M_{v(1,7)}(\vw) =M&\biggl[{a;h;c,d,\atop e,f,g,b}\biggr]  \nr\leftrightarrow
L_6(\vx)=
L&\biggl[{e,f,g,1+a-c-d;\atop e+f+g-a;1+a-c,1+a-d}\biggr]
\nrs
M_{v(1,6)}(\vw)=M&\biggl[{2c-a;c+h-a;c,c+d-a,\atop c+e-a,c+f-a,c+g-a,c+b-a}\biggr] \nr\leftrightarrow
L_5(\vx)=L&\biggl[{c+e-a,c+f-a,c+g-a,1-d;\atop c+e+f+g-2a;1+c-a,1+c-d}\biggr]
\nrs
M_{v(1,5)}(\vw)=M&\biggl[{2d-a;d+h-a;d,d+c-a,\atop d+e-a,d+f-a,d+g-a,d+b-a}\biggr] \nr\leftrightarrow
L_4(\vx)=L&\biggl[{d+e-a,d+f-a,d+g-a,1-c;\atop d+e+f+g-2a;1+d-a,1+d-c}\biggr]
\nrs
M_{v(1,4)}(\vw)=M&\biggl[{2e-a;e+h-a;e,e+c-a,\atop e+d-a,e+f-a,e+g-a,e+b-a}\biggr] \nr\leftrightarrow
L_3(\vx)=L&\biggl[{e+d-a,e+f-a,e+g-a,1-c;\atop e+d+f+g-2a;1+e-a,1+e-c}\biggr]
;
\end{align}
\begin{align}
M_{v(1,3)}(\vw)=M&\biggl[{2f-a;f+h-a;f,f+c-a,\atop f+d-a,f+e-a,f+g-a,f+b-a}\biggr] \nr\leftrightarrow
L_2(\vx)=L&\biggl[{f+d-a,f+e-a,f+g-a,1-c;\atop f+d+e+g-2a;1+f-a,1+f-c}\biggr]
\label{v13ex}\nrs
M_{v(1,2)}(\vw)=M&\biggl[{2g-a;g+h-a;g,g+c-a,\atop g+d-a,g+e-a,g+f-a,g+b-a}\biggr] \nr\leftrightarrow
L_1(\vx)=L&\biggl[{g+d-a,g+e-a,g+f-a,1-c;\atop g+d+e+f-2a;1+g-a,1+g-c}\biggr]
\nrs
M_{-v(0,7)}(\vw)=M&\biggl[{1-a;1-b;1-c,1-d,\atop 1-e,1-f,1-g,1-h}\biggr] \nr\leftrightarrow
L_{\overline6}(\vx)=L&\biggl[{1-e,1-f,1-g,c+d-a;\atop 2+a-e-f-g;1+c-a,1+d-a}\biggr]
\nrs
M_{-v(0,6)}(\vw)=M&\biggl[{1+a-2c;1+a-c-b;1-c,1+a-c-d,\atop 1+a-c-e,1+a-c-f,1+a-c-g,1+a-c-h}\biggr] \nr\leftrightarrow
L_{\overline5}(\vx)=L&\biggl[{1+a-c-e,1+a-c-f,1+a-c-g,d;\atop 2+2a-c-e-f-g;1+a-c,1+d-c}\biggr]
\nrs
M_{-v(0,5)}(\vw)=M&\biggl[{1+a-2d;1+a-d-b;1-d,1+a-d-c,\atop 1+a-d-e,1+a-d-f,1+a-d-g,1+a-d-h}\biggr] \nr\leftrightarrow
L_{\overline4}(\vx)=L&\biggl[{1+a-d-e,1+a-d-f,1+a-d-g,c;\atop 2+2a-d-e-f-g;1+a-d,1+c-d}\biggr]
\nrs
M_{-v(0,4)}(\vw)=M&\biggl[{1+a-2e;1+a-e-b;1-e,1+a-e-c,\atop 1+a-e-d,1+a-e-f,1+a-e-g,1+a-e-h}\biggr] \nr\leftrightarrow
L_{\overline3}(\vx)=L&\biggl[{1+a-e-d,1+a-e-f,1+a-e-g,c;\atop 2+2a-e-d-f-g;1+a-e,1+c-e}\biggr]
\nrs
M_{-v(0,3)}(\vw)=M&\biggl[{1+a-2f;1+a-f-b;1-f,1+a-f-c,\atop 1+a-f-d,1+a-f-e,1+a-f-g,1+a-f-h}\biggr] \nr\leftrightarrow
L_{\overline2}(\vx)=L&\biggl[{1+a-f-d,1+a-f-e,1+a-f-g,c;\atop 2+2a-f-d-e-g;1+a-f,1+c-f}\biggr]
\nrs
M_{-v(0,2)}(\vw)=M&\biggl[{1+a-2g;1+a-g-b;1-g,1+a-g-c,\atop 1+a-g-d,1+a-g-e,1+a-g-f,1+a-g-h}\biggr] \nr\leftrightarrow
L_{\overline1}(\vx)=L&\biggl[{1+a-g-d,1+a-g-e,1+a-g-f,c;\atop 2+2a-g-d-e-f;1+a-g,1+c-g}\biggr].
\end{align}

\subsection*{\boldmath Orbit $\mathscr{O}_3$ $(J)$}
\begin{align}
M_{v(0,1)}(\vw)=M&\biggl[{  2+a-c-d-e-f;2+2a-c-d-e-f-g;1-c, 1-d, 
\atop
1-e,1-f,b+g-a,h+g-a}\biggr]\nr\leftrightarrow
J_{p_0}(\vx)=J&\biggl[{2+2a-c-d-e-f-g;1-c,1-d,1+a-c-d;\atop 2+a-c-d-e,2+a-c-d-f,2+a-c-d-g}\biggr];\label{v01}\\
M_{-v(3,4)}(\vw)=M&\biggl[{1+a-2e;1+a-e-f;1-e,1+a-e-b,\atop 1+a-e-c,1+a-e-d,1+a-e-g,1+a-e-h}\biggr] \nr\leftrightarrow
J_{p_1}(\vx)=J&\biggl[{1+a-e-f;1+a-e-c,1+a-e-d,g;\atop 2+2a-c-d-e-f,1+a-e,1+g-e}\biggr]
\nrs
M_{-v(2,4)}(\vw)=M&\biggl[{1+a-2e;1+a-e-g;1-e,1+a-e-b,\atop 1+a-e-c,1+a-e-d,1+a-e-f,1+a-e-h}\biggr] \nr\leftrightarrow
J_{p_2}(\vx)=J&\biggl[{1+a-e-g;1+a-e-c,1+a-e-d,f;\atop 2+2a-c-d-e-g,1+a-e,1+f-e}\biggr]
\nrs
M_{-v(2,3)}(\vw)=M&\biggl[{1+a-2f;1+a-f-g;1-f,1+a-f-b,\atop 1+a-f-c,1+a-f-d,1+a-f-e,1+a-f-h}\biggr] \nr\leftrightarrow
J_{p_3}(\vx)=J&\biggl[{1+a-f-g;1+a-f-c,1+a-f-d,e;\atop 2+2a-c-d-f-g,1+a-f,1+e-f}\biggr]
\nrs
M_{-v(6,7)}(\vw)=M&\biggl[{1-a;1-c;1-b,1-d,\atop 1-e,1-f,1-g,1-h}\biggr] \nr\leftrightarrow
J_{p_4}(\vx)=J&\biggl[{1-c;1-d,1-e,f+g-a;\atop 2+a-c-d-e,1+f-a,1+g-a}\biggr]
\nrs
M_{v(2,5)}(\vw)=M&\biggl[{2d-a;d+g-a;d,d+b-a,\atop d+c-a,d+e-a,d+f-a,d+h-a}\biggr] \nr\leftrightarrow
J_{p_5}(\vx)=J&\biggl[{d+g-a;d+c-a,d+e-a,1-f;\atop c+d+e+g-2a,1+d-a,1+d-f}\biggr]
\nrs
M_{v(3,5)}(\vw)=M&\biggl[{2d-a;d+f-a;d,d+b-a,\atop d+c-a,d+e-a,d+g-a,d+h-a}\biggr] \nr\leftrightarrow
J_{p_6}(\vx)=J&\biggl[{d+f-a;d+c-a,d+e-a,1-g;\atop c+d+e+f-2a,1+d-a,1+d-g}\biggr]
\nrs
M_{v(4,5)}(\vw)=M&\biggl[{2d-a;d+e-a;d,d+b-a,\atop d+c-a,d+f-a,d+g-a,d+h-a}\biggr] \nr\leftrightarrow
J_{p_7}(\vx)=J&\biggl[{d+e-a;d+c-a,d+f-a,1-g;\atop c+d+e+f-2a,1+d-a,1+d-g}\biggr]
;\end{align}\begin{align}
M_{-v(5,7)}(\vw)=M&\biggl[{1-a;1-d;1-b,1-c,\atop 1-e,1-f,1-g,1-h}\biggr] \nr\leftrightarrow
J_{p_8}(\vx)=J&\biggl[{1-d;1-c,1-e,f+g-a;\atop 2+a-c-d-e,1+f-a,1+g-a}\biggr]
\nrs
M_{v(2,6)}(\vw)=M&\biggl[{2c-a;c+g-a;c,c+b-a,\atop c+d-a,c+e-a,c+f-a,c+h-a}\biggr] \nr\leftrightarrow
J_{p_9}(\vx)=J&\biggl[{c+g-a;c+d-a,c+e-a,1-f;\atop c+d+e+g-2a,1+c-a,1+c-f}\biggr]
\nrs
M_{v(3,6)}(\vw)=M&\biggl[{2c-a;c+f-a;c,c+b-a,\atop c+d-a,c+e-a,c+g-a,c+h-a}\biggr] \nr\leftrightarrow
J_{p_{10}}(\vx)=J&\biggl[{c+f-a;c+d-a,c+e-a,1-g;\atop c+d+e+f-2a,1+c-a,1+c-g}\biggr]
\nrs
M_{v(4,6)}(\vw)=M&\biggl[{2c-a;c+e-a;c,c+b-a,\atop c+d-a,c+f-a,c+g-a,c+h-a}\biggr] \nr\leftrightarrow
J_{p_{11}}(\vx)=J&\biggl[{c+e-a;c+d-a,c+f-a,1-g;\atop c+d+e+f-2a,1+c-a,1+c-g}\biggr]
\nrs
M_{-v(5,6)}(\vw)=M&\biggl[{1+a-2c;1+a-c-d;1-c,1+a-c-b,\atop 1+a-c-e,1+a-c-g,1+a-c-f,1+a-c-h}\biggr] \nr\leftrightarrow
J_{p_{12}}(\vx)=J&\biggl[{1+a-c-d;1+a-c-e,1+a-c-g,f;\atop 2+2a-c-d-e-g,1+a-c,1+f-c}\biggr]
\nrs
M_{v(2,7)}(\vw)=M&\biggl[{a;g;b,c,\atop d,e,f,h}\biggr] \nr\leftrightarrow
J_{p_{13}}(\vx)=J&\biggl[{g;c,d,1+a-e-f;\atop c+d+g-a,1+a-e,1+a-f}\biggr]
\nrs
M_{v(3,7)}(\vw)=M&\biggl[{a;f;b,c,\atop d,e,g,h}\biggr] \nr\leftrightarrow
J_{p_{14}}(\vx)=J&\biggl[{f;c,d,1+a-e-g;\atop c+d+f-a,1+a-e,1+a-g}\biggr]
\nrs
M_{v(4,7)}(\vw)=M&\biggl[{a;e;b,c,\atop d,f,g,h}\biggr] \nr\leftrightarrow
J_{p_{15}}(\vx)=J&\biggl[{e;c,d,1+a-f-g;\atop c+d+e-a,1+a-f,1+a-g}\biggr]
\nrs
M_{-v(0,1)}(\vw)=M&\biggl[{c+d+e+f-a-1;c+d+e+f+g-2a-1;c,d, 
\atop
e,f,1+a-b-g,1+a-h-g}\biggr] \nr\leftrightarrow
J_{n_0}(\vx)=J&\biggl[{c+d+e+f+g-2a-1;c,d,c+d-a;\atop c+d+e-a,c+d+f-a,c+d+g-a}\biggr]
;
\end{align}\begin{align}
M_{v(3,4)}(\vw)=M&\biggl[{2e-a;e+f-a;e,e+b-a,\atop e+c-a,e+d-a,e+g-a,e+h-a}\biggr] \nr\leftrightarrow
J_{n_1}(\vx)=J&\biggl[{e+f-a;e+c-a,e+d-a,1-g;\atop c+d+e+f-2a,1+e-a,1+e-g}\biggr]
\nrs
M_{v(2,4)}(\vw)=M&\biggl[{2e-a;e+g-a;e,e+b-a,\atop e+c-a,e+d-a,e+f-a,e+h-a}\biggr] \nr\leftrightarrow
J_{n_2}(\vx)=J&\biggl[{e+g-a;e+c-a,e+d-a,1-f;\atop c+d+e+g-2a,1+e-a,1+e-f}\biggr]
\nrs
M_{v(2,3)}(\vw)=M&\biggl[{2f-a;f+g-a;f,f+b-a,\atop f+c-a,f+d-a,f+e-a,f+h-a}\biggr] \nr\leftrightarrow
J_{n_3}(\vx)=J&\biggl[{f+g-a;f+c-a,f+d-a,1-e;\atop c+d+f+g-2a,1+f-a,1+f-e}\biggr]
\nrs
M_{v(6,7)}(\vw)=M&\biggl[{a;c;b,d,\atop e,f,g,h}\biggr] \nr\leftrightarrow
J_{n_4}(\vx)=J&\biggl[{c;d,e,1+a-f-g;\atop c+d+e-a,1+a-f,1+a-g}\biggr]
;\label{v67}\\
M_{-v(2,5)}(\vw)=M&\biggl[{1+a-2d;1+a-d-g;1-d,1+a-d-b,\atop 1+a-d-c,1+a-d-e,1+a-d-f,1+a-d-h}\biggr] \nr\leftrightarrow
J_{n_5}(\vx)=J&\biggl[{1+a-d-g;1+a-d-c,1+a-d-e,f;\atop 1+2a-c-d-e-g,1+a-d,1+f-d}\biggr]
\nrs
M_{-v(3,5)}(\vw)=M&\biggl[{1+a-2d;1+a-d-f;1-d,1+a-d-b,\atop 1+a-d-c,1+a-d-e,1+a-d-g,1+a-d-h}\biggr] \nr\leftrightarrow
J_{n_6}(\vx)=J&\biggl[{1+a-d-f;1+a-d-c,1+a-d-e,g;\atop 1+2a-c-d-e-f,1+a-d,1+g-d}\biggr]
\nrs
M_{-v(4,5)}(\vw)=M&\biggl[{1+a-2d;1+a-d-e;1-d,1+a-d-b,\atop 1+a-d-c,1+a-d-f,1+a-d-g,1+a-d-h}\biggr] \nr\leftrightarrow
J_{n_7}(\vx)=J&\biggl[{1+a-d-e;1+a-d-c,1+a-d-f,g;\atop 1+2a-c-d-e-f,1+a-d,1+g-d}\biggr]
\nrs
M_{v(5,7)}(\vw)=M&\biggl[{a;d;b,c,\atop e,f,g,h}\biggr] \nr\leftrightarrow
J_{n_8}(\vx)=J&\biggl[{d;c,e,1+a-f-g;\atop c+d+e-a,1+a-f,1+a-g}\biggr]
\nrs
M_{-v(2,6)}(\vw)=M&\biggl[{1+a-2c;1+a-c-g;1-c,1+a-c-b,\atop 1+a-c-d,1+a-c-e,1+a-c-f,1+a-c-h}\biggr] \nr\leftrightarrow
J_{n_9}(\vx)=J&\biggl[{1+a-c-g;1+a-c-d,1+a-c-e,f;\atop 1+2a-c-d-e-f,1+a-c,1+f-c}\biggr]
;\end{align}\begin{align}
M_{-v(3,6)}(\vw)=M&\biggl[{1+a-2c;1+a-c-f;1-c,1+a-c-b,\atop 1+a-c-d,1+a-c-e,1+a-c-g,1+a-c-h}\biggr] \nr\leftrightarrow
J_{n_{10}}(\vx)=J&\biggl[{1+a-c-f;1+a-c-d,1+a-c-e,g;\atop 1+2a-c-d-e-g,1+a-c,1+g-c}\biggr]
\nrs
M_{-v(4,6)}(\vw)=M&\biggl[{1+a-2c;1+a-c-e;1-c,1+a-c-b,\atop 1+a-c-d,1+a-c-f,1+a-c-g,1+a-c-h}\biggr] \nr\leftrightarrow
J_{n_{11}}(\vx)=J&\biggl[{1+a-c-e;1+a-c-d,1+a-c-f,g;\atop 1+2a-c-d-f-g,1+a-c,1+g-c}\biggr]
;\label{minusv46ex}\\
M_{v(5,6)}(\vw)=M&\biggl[{2c-a;c+d-a;c,c+b-a,\atop c+e-a,c+g-a,c+f-a,c+h-a}\biggr] \nr\leftrightarrow
J_{n_{12}}(\vx)=J&\biggl[{c+d-a;c+e-a,c+g-a,1-f;\atop c+d+e+g-2a,1+c-a,1+c-f}\biggr]
\nrs
M_{-v(2,7)}(\vw)=M&\biggl[{1-a;1-g;1-b,1-c,\atop 1-d,1-e,1-f,1-h}\biggr] \nr\leftrightarrow
J_{n_{13}}(\vx)=J&\biggl[{1-g;1-c,1-d,e+f-a;\atop 2+a-c-d-g,1+e-a,1+f-a}\biggr]
\nrs
M_{-v(3,7)}(\vw)=M&\biggl[{1-a;1-f;1-b,1-c,\atop 1-d,1-e,1-g,1-h}\biggr] \nr\leftrightarrow
J_{n_{14}}(\vx)=J&\biggl[{1-f;1-c,1-d,e+g-a;\atop 2+a-c-d-f,1+e-a,1+g-a}\biggr]
\nrs
M_{-v(4,7)}(\vw)=M&\biggl[{1-a;1-e;1-b,1-c,\atop 1-d,1-f,1-g,1-h}\biggr] \nr\leftrightarrow
J_{n_{15}}(\vx)=J&\biggl[{1-e;1-c,1-d,f+g-a;\atop 2+a-c-d-e,1+f-a,1+g-a}\biggr].
\end{align}

\end{document}